\documentclass{aims}
\usepackage{amsmath}
  \usepackage{paralist}
  \usepackage{graphics} 
  \usepackage{epsfig} 
\usepackage{graphicx}  \usepackage{epstopdf}
 \usepackage[colorlinks=true]{hyperref}
\hypersetup{urlcolor=blue, citecolor=red}

\def\currentvolume{}
 \def\currentissue{}
  \def\currentyear{}
   \def\currentmonth{}
    \def\ppages{}
     \def\DOI{}

\newtheorem{theorem}{Theorem}[section]
\newtheorem{corollary}{Corollary}

\newtheorem{lemma}[theorem]{Lemma}
\newtheorem{proposition}{Proposition}

\theoremstyle{definition}

\newtheorem{remark}{Remark}

\title[Non-diagonal metric on a product riemannian manifold] 
      {Non-diagonal metric on a product riemannian manifold}

\author[Rafik Nasri]{}


  \email{rmag1math@yahoo.fr}
  \email{djaamustapha@live.com}

\begin{document}
\maketitle

\maketitle


\centerline{\scshape Rafik Nasri and Mustapha Djaa}
\medskip
{\footnotesize
 \centerline{Laboratory of Geometry, Analysis, Control and Applications}
   \centerline{Universit\'e de Sa\"{i}da}
   \centerline{BP138, En-Nasr, 20000 Sa\"ida, Algeria}
}

\bigskip

 \centerline{(Communicated by the associate editor name)}

\begin{abstract}
In this paper, we construct the symmetric tensor field $G_{f_1f_2}$ and $h_{f_1f_2}$
on a product manifold and we give conditions under which $G_{f_1f_2}$ becomes
a metric tensor, theses tensors fields will be called the generalized warped
product, and then we develop an expression of curvature for the connection of
the generalized warped product in relation to those corresponding analogues of
its base and fiber and warping functions. By constructing a frame field in
$M_1\times_{f_1f_2}M_2$ with respect to the Riemannian metric $G_{f_1f_2}$ and
$h_{f_1f_2}$, then we calculate the Laplacian$-$Beltrami operator of a function
on a generalized warped product which may be expressed in terms of the local
restrictions of the functions to the base and fiber. Finally, we conclude some
interesting relationships between the geometry of the couples $(M_1,g_1)$ and
 $(M_2,g_2)$ and that of $(M_1\times M_2,h_{f_1f_2})$.

\end{abstract}


\section{Introduction}
The warped product provides a way to construct new pseudo-rieman
nian manifolds from the given ones, see \cite{B. O'Neill},\cite{bishop} and
\cite{Beem2}. This construction has useful applications in general
relativity, in the study of cosmological models and black holes. It
generalizes the direct product in the class of pseudo-Riemannian
manifolds and it is defined as follows. Let $(M_1,g_1)$ and
$(M_2,g_2)$ be two pseudo-Riemannian manifolds and let
$f_1:M_1\longrightarrow\mathbb{R}^*$ be a positive smooth function on
$M_1$, the warped product of $(M_1,g_1)$ and
$(M_2,g_2)$ is the product manifold $M_1\times M_2$ equipped
with the metric tensor $g_{f_1}:=\pi_1^*g_1
+(f\circ \pi_1)^2\pi_2^*g_2$, where $\pi_1$ and
$\pi_2$ are the projections of $M_1\times M_2$ onto $M_1$ and
$M_2$ respectively. The manifold $M_1$ is called the base of
$(M_1\times M_2,g_{f_1})$ and $M_2$ is called the fiber. The function $f_1$
is called the warping function.  \\
The doubly warped product is  construction in the class of pseudo-Riemannian
manifolds generalizing the warped product and the direct product, it is obtained by homothetically
distorting the geometry of each base $M_{_1}\times\{q\}$ and each fiber $\{p\}\times M_{_2}$ to get a new
"doubly warped" metric tensor on the product manifold and defined as follows.
For $i\in \{1,2\}$, let $M_i$ be a pseudo-Riemannian manifold equipped with metric
$g_i$, and $f_{_i}: M_i\rightarrow \mathbb{R}^*$ be a positive smooth function on $M_i$.
 The well-know notion of doubly warped product
manifold $M_{_1}\times _{_{f_1f_2}}M_{_2}$ is defined as the product manifold $M=M_{_1}\times M_{_2}$
equipped with pseudo-Riemannian metric which is denoted by $g_{_{f_1f_2}}$, given by
$$
g_{_{f_1f_2}}=(f_2\circ\pi_2)^2\pi_1^* g_{_{_1}}+(f_1\circ\pi_1)^2\pi_2^*g_{_{_2}}.
$$
When the warping functions $f_1=1$ or $f_2=1$ we obtain a warped product or direct product. \\

The paper is organized as follows. In section 2, we collect the basic material about
Levi-Civita connection, horizontal and vertical lifts.
In section 3, we consider the metric tensors $g_{_1}$ and $g_{_2}$ on manifolds $M_{_1}$ and $M_{_2}$
respectively and, for a smooth function $f_{_i}$ on $M_{_i}$, $i=1,2$, we define the symmetric tensors
fields $G_{_{f_{_1}f_{_2}}}$  and $h_{_{f_{_1}f_{_2}}}$ on $M_{_1}\times M_{_1}$ relative to
$g_{_1}$, $g_{_2}$ and the warping functions $f_{_1}$, $f_{_1}$, then we give the condition under
which $G_{_{f_{_1}f_{_2}}}$ becomes a metric tensor, this tensor field will be referred to as
the generalized warped product metric, next, we define also its cometric and we compute the gradients
of the lifts of $f_{_1}$, $f_{_2}$. Morever, by constructing a frame field in $M_1\times_{f_1f_2}M_2$
with respect to the Riemannian metric $G_{f_1f_2}$, then we calculate the Laplacian$-$Beltrami operator
of a function on a generalized warped product which may be expressed in terms of the local restrictions
of the functions to the base and fiber. To end this section, we conclude with some important relationships
related to the harmonicity of function. In the final section, we compute the curvatures of generalized
warped product $h_{_{f_{_1}f_{_2}}}$ and we conclude with some important relationships between the geometry
of the triples  $(M_1,g_1)$, $(M_2,g_2)$ and that of $(M_1\times M_2,h_{_{f_{_1}f_{2}}})$.

\section{Preliminaries}

\subsection{Horizontal and vertical lifts}
Throughout this paper $M_{1}$ and $M_{2}$ will be respectively
$m_{1}$ and $m_{2}$ dimensional manifolds, $M_1\times M_2$ the
product manifold with the natural product coordinate system and
$\pi_1:M_{1}\times M_{2}\rightarrow M_{1}$ and $\pi_2
:M_{1}\times M_{2}\rightarrow M_{2}$ the usual projection maps.

We recall briefly how the calculus on the product manifold $M_1
\times M_2$ derives from that of $M_1$ and $M_2$ separately. For
details see \cite{B. O'Neill}.

Let $\varphi _{1}$ in $C^{\infty }(M_{1})$. The horizontal lift of
$\varphi_{1}$ to $M_{1}\times M_{2}$ is $\varphi_{1}^{h}=\varphi
_{1}\circ \pi_1$. One can define the horizontal lifts of
tangent vectors as follows. Let $p_1\in M_1$ and let $X_{p_1}\in
T_{p_1}M_{1}$. For any $p_2\in M_{2}$ the horizontal lift
 of $X_{p_1}$ to $(p_1,p_2)$ is the unique tangent vector $X_{(p_1,p_2)}^{h}$
in $T_{(p_1,p_2)}(M_{1}\times M_2)$ such that
$ d_{(p_1,p_2)}\pi_1(X_{(p_1,p_2)}^{h})=X_{p_1}$ and $d_{(p_1,p_2)}\pi_2(X_{(p_1,p_2)}^{h})=0.$\\
We can also define the horizontal lifts of vector fields as follows.
Let $X_1\in \Gamma (TM_{1})$. The horizontal lift of $X_1$ to
$M_{1}\times M_{2}$ is the vector field $X_1^{h}\in \Gamma
(T(M_{1}\times M_{2}))$ whose value at each $(p_1,p_2)$ is the horizontal
lift of the tangent vector $(X_1){p_1}$ to $(p_1,p_2)$. For $(p_1,p_2)\in M_1\times M_2$,
we will denote the set of the horizontal lifts to $(p_1,p_2)$ of all the tangent
vectors of $M_{1}$ at $p_1$ by $L(p_1,p_2)(M_{1})$. We will denote the set of the
horizontal lifts of all vector fields on $M_{1}$ by $\mathfrak{L}(M_{1})$.

The vertical lift $\varphi_2^v$ of a function $\varphi_2\in
C^{\infty}(M_2)$ to $M_1\times M_2$ and the vertical lift $X_2^v$ of
a vector field $X_2\in \Gamma (TM_{2})$ to $M_1\times M_2$ are
defined in the same way using the projection $\pi_2$. Note that
the spaces $\mathfrak{L}(M_{1})$ of the horizontal lifts and
$\mathfrak{L}(M_{2})$ of the vertical lifts are vector subspaces of
$\Gamma (T(M_{1}\times M_{2}))$ but neither is invariant under
multiplication by arbitrary functions $\varphi \in C^{\infty
}(M_{1}\times M_{2})$.\\

Observe that if $\{\frac{\partial}{\partial x_1},\ldots,\frac{\partial}
{\partial x_{m_1}}\}$ is the local basis of the vector fields (resp.
$\{dx_1,\ldots,dx_{m_1}\}$ is the local basis of $1$-forms ) relative to a chart
$(U,\Phi)$ of $M_1$ and $\{\frac{\partial}{\partial y_1},\ldots,\frac{\partial}
{\partial y_{m_2}}\}$ is the local basis of the vector fields (resp. $\{dy_1,
\ldots,dy_{m_2}\}$ the local basis of the $1$-forms) relative to a chart $(V,\Psi)$
of $M_2$, then $\{(\frac{\partial}{\partial x_1})^h,\ldots,(\frac{\partial}
{\partial x_{m_1}})^h,(\frac{\partial}{\partial y_1})^v,\\
\ldots,(\frac{\partial}{\partial y_{m_2}})^v \}$ is the local
basis of the vector fields (resp. $\{(dx_1)^h,\ldots,(dx_{m_1})^h,
(dy_1)^v,\\
\ldots,(dy_{m_2})^v\}$ is the local basis of the $1$-forms)  relative to the chart $(U\times
V,\Phi\times \Psi)$ of $M_1\times M_2$.\\

The following lemma will be useful later for our computations.
\begin{lemma} \label{lift} $\;$
\begin{enumerate}
\item Let $\varphi_i\in C^\infty(M_i)$, $X_i,Y_i\in \Gamma (TM_{i})$
and $\alpha _{i}\in \Gamma (T^* M_{i})$, $i=1,2$. Let
$\varphi=\varphi_1^h+\varphi_2^v$, $X=X_{1}^{h}+X_{2}^{v}$ and
$\alpha ,\beta \in \Gamma (T^*(M_{1}\times M_{2}))$. Then
\begin{enumerate}
\item[i/] For all $(i,I)\in \{(1,h),(2,v)\}$, we have
$$
X_i^I(\varphi)=X_i(\varphi_i)^I,\quad [X,Y_i^I]=[X_i,Y_i]^I \quad
\textrm{ and } \quad  \alpha _{i}^{I}(X)=\alpha_{i}(X_{i})^{I}.
$$
\item[ii/] If for all $(i,I)\in \{(1,h),(2,v)\}$ we have $\alpha
(X_{i}^{I})=\beta (X_{i}^{I})$, then $\alpha =\beta$.
\end{enumerate}
\item Let $\omega_i$ and $\eta_i$ be $r$-forms on $M_i$, $i=1,2$.
Let $\omega=\omega_1^h+\omega_2^v$ and $\eta=\eta_1^h+\eta_2^v$. We
have
$$
d\omega=(d\omega_1)^h+(d\omega_2)^v \quad \textrm{ and } \quad
\omega \wedge \eta=(\omega_1\wedge \eta_1)^h+(\omega_2\wedge
\eta_2)^v.
$$
\end{enumerate}
\end{lemma}
\begin{proof}
See \cite{Nas}.
\end{proof}
\begin{remark}\label{rem dpi}
    Let $X$ be a vector field on $M_1\times M_2$, such that $d\pi_1(X)=\varphi(X_1\circ\pi_1)$
and $d\pi_2(X)=\phi(X_2\circ\pi_2)$, then $X=\varphi X_1^h+\phi X_2^v$.
\end{remark}

\section{ About generalized warped products}
\subsection{The generalized warped product.}
let $\psi: M\rightarrow N$ be a smooth map between smooth manifolds and let $g$ be a metric on
 $k$-vector bundle $(F,P_F)$ over $N$. The metric
$g^{\psi}: \Gamma(\psi^{-1}F)\times \Gamma(\psi^{-1}F)\rightarrow C^{\infty}(M)$ on the pull-back
$(\psi^{-1}F,P_{\psi^{-1}F})$ over $M$  is defined by
$$
g^{\psi}(U,V)(p)=g_{\psi(p)}(U_p,V_p),~ ~\forall~ U,V\in \Gamma(\psi^{-1}F),~ p\in M.
$$
Given a linear connection $\nabla^N$ on $k$-vector bundle $(F,P_F)$ over $N$, the pull-back
connection $\nabla{\hskip-0.3cm^{^{^{\psi}}}}$ is the unique linear connection on the pull-back
$(\psi^{-1}F,P_{\psi^{-1}F})$ over $M$ such that
\begin{equation}
  \nabla{\hskip-0.3cm^{^{^{\psi}}}}_{X} \big(W\circ\psi\big)=\nabla^N{\hskip -0.4cm_{_{_{d\psi(X)}}}}\hskip -0.3cm W,
   \hskip 0.4cm \forall W\in\Gamma(F),~ \forall X\in\Gamma(TM).
\end{equation}
Further, let $U\in\psi^{-1}F$ and let $p\in M$, $X\in\Gamma(TM)$. Then
\begin{equation}
   (\nabla{\hskip-0.3cm^{^{^{\psi}}}}_{X} U)(p)=(\nabla^N{\hskip -0.4cm_{_{_{d{\!_{_{_p}}}
   \!\!\!\psi(X_{_p}\!)}}}}\hskip -0.4cm\widetilde{U})(\psi(p)),
\end{equation}
where $\widetilde{U}\in\Gamma(F)$ with $\widetilde{U}\circ\psi=U$.\\
Now, let $\pi_i$, i=1,2, be the usual projection of $M_1\times M_2$ onto $M_i$,
given a linear connection $\nabla{\hskip-0.2cm^{^{^{i}}}}$ on vector bundle $TM_i$, the pull-back
connection $\nabla{\hskip-0.4cm^{^{^{\pi_i}}}}$ is the unique linear connection on the pull-back
$M_1\times M_2\rightarrow\pi_i^{-1}(TM_i)$ such that for each $Y_i\in\Gamma(TM_i)$,
 $X\in\Gamma(TM_1\times M_2)$
\begin{equation}
  \nabla{\hskip-0.3cm^{^{^{\pi_i}}}}_{X} \big(Y_i\circ\pi_i\big)=\nabla{\hskip-0.2cm^{^{^{i}}}}
  {\hskip -0.05cm_{_{_{_{d\pi_i(X)}}}}}\hskip -0.5cmY_i.
\end{equation}
Further, let $(p_1,p_2)\in M_1\times M_2$, $U\in\pi_i^{-1}(TM)$ and $X\in\Gamma(TM_1\times M_2)$. Then
\begin{equation}
   (\nabla{\hskip-0.3cm^{^{^{\pi_i}}}}_{X}U)(p_1,p_2)=\big(\nabla{\hskip-0.2cm^{^{^{i}}}}
   _{d{\hskip -0.2cm_{_{_{(p_1,p_2)}}}}\hskip -0.6cm\pi_i(X_{(p_1,p_2)})}{\widetilde{U}}\big)(p_i),
\end{equation}

Now, we construct a symmetric tensor fiels on product manifold and give the condition
under which it becomes a tensor metric.\\
Let $c$ be an arbitrary real number and let $g_i$, $(i=1,2)$ be a Riemannian metric
tensors on $M_i$. Given a smooth positive function $f_i$ on $M_i$,
we define a symmetric tensor field on $M_1\times M_2$ by
\begin{equation}\label{generalized metric}
        G_{_{f_1,f_2}}=(f_2^v)^2\pi_1^*g_{_{_1}}+(f_1^h)^2\pi_2^*g_{_{_2}}
        +cf_1^hf_2^v df_1^h\odot df_2^v.
\end{equation}
Where $\pi_i$, $(i=1,2)$ is the projection of $M_{_1}\times M_{_2}$ onto $M_{_i}$ and
$$
df_1^h\odot df_2^v = df_1^h\otimes df_2^v + df_2^v\otimes df_1^h.
$$
For all $X,Y\in\Gamma(TM_1\times M_2)$
$$
\begin{array}{rl}
 G_{_{f_1,f_2}}(X,Y)&=(f_2^v)^2g_{_{_1}}^{\pi_1}(d\pi_1(X),d\pi_1(Y))
 +(f_1^h)^2g_{_{_2}}^{\pi_2}(d\pi_2(X),d\pi_2(Y)) \\
 &\\
&+cf_1^hf_2^v\left(X(f_1^h)Y(f_2^v)+X(f_2^v)Y(f_1^h))\right).
\end{array}
$$
It is the unique tensor fields such that for any $X_i,Y_i\in\Gamma(TM_i)$, $(i=1,2)$
\begin{equation} \label{equivalent generalized warped}
 G_{_{f_1f_2}}(X_i^I,Y_k^K)= \left\{
  \begin{array}{ccl}
   (f_{3-i}^{J})^2g_i(X_i,Y_i)^I,& &\text{if} ~(i,I)=(k,K) \\
&&\\
    cf_i^If_k^KX_i(f_i)^IY_k(f_k)^K,& & \text{otherwise}
  \end{array}
\right.
\end{equation}
where $(i,I),(k,K),(3-i,J)\in \{(1,h),(2,v)\}$.\\
We call $G_{f_1,f_2}$ the generalized warped product relative to $g_1, g_2$ and
the warping functions $f_1,f_2$.\\
If either $f_1\equiv 1$ or $f_2\equiv 1$ but not both, then we obtain a singly warped
product. If both $f_1\equiv 1$ and $f_2\equiv 1$, then we have a product manifold.
If neither $f_1$ nor $f_2$ is constant and $c=0$, then we have a nontrivial doubly warped product.
If neither $f_1$ nor $f_2$ is constant and $c\neq 0$, then we have a nontrivial generalized warped product.

Now, Let us assume that $(M_i,g_i)$, $(i=1,2)$ is a smooth connected Riemannian manifold.
The following proposition provides a necessary and sufficient condition for a symmetric tensor field
 $G_{f_1,f_2}$ of type $(0,2)$ of two Riemannian metrics to be a Riemannian metric.

 \begin{proposition}\label{condition generalized warped}
Let $(M_i,g_i)$, $(i=1,2)$ be a Riemannian manifold
and let $f_i$ be a positive smooth function on $M_i$ and $c$ be an arbitrary real number.
Then the symmetric tensor field
$G_{f_1f_2}$is Riemannian metric on $M_1\times M_2$ if and only if
\begin{equation}\label{condition of metric}
   0 \leq c^2g_1(gradf_1,gradf_1)^h g_2(gradf_2,gradf_2)^v<1.
\end{equation}
\end{proposition}
\begin{proof}

Let $\{e_{_1},...,e_{_{m_{_1}}}\}$ and $\{e_{_{m_1+1}},...,e_{_{m_1+m_2}}\}$ be a local, orthonormal
basis of the vector fields with respect to $g_{_1}$ and $g_{_2}$ on an open $O_{_1}\subset M_{_1}$
and $O_{_2}\subset M_{_2}$ respectively. The matrix of $G_{f_1f_2}$ relative to
$$
\{v_{_1}=\frac{1}{f_{_2}^v}e_{_1}^h,...,v_{_{m_1}}=\frac{1}{f_{_2}^v}e_{_{m_1}}^h,
v_{_{m_1}+1}=\frac{1}{f_{_1}^h}e_{_{m_1}+1}^v,...,v_{_{m_1+m_2}}=\frac{1}{f_{_1}^h}e_{_{m_1+m_2}}^v\}
$$
has the form
\begin{equation}\label{matrix1}
 \left(
                  \begin{array}{cc}
                    D_1 & cf_1^hf_2^vE \\
                    cf_1^hf_2^v ~^tE & D_2\\
                  \end{array}
                \right).
\end{equation}
Where $D_1=(f_2^v)^2I_{m_1}$,~ $D_2=(f_1^h)^2I_{m_2}$ and
$$
E=\left(
\begin{array}{ccc}
e_1(f_1)^he_{m_1+1}(f_2)^v & \cdots & e_1(f_1)^he_{m_1+m_2}(f_2)^v \\
\vdots& \ddots & \vdots \\
e_{m_1}(f_1)^he_{m_1+1}(f_2)^v  & \cdots & e_{m_1}(f_1)^he_{m_1+m_2}(f_2)^v
\end{array}
\right)
$$
We can write the matrix (\ref{matrix1}) as
\begin{equation}
   \left(
                          \begin{array}{cc}
                            I_{m_1} & O \\
                             cf_1^hf_2^v ~^tED_1^{-1} & -(cf_1^hf_2^v)^2~^tED_1^{-1}E+D_2 \\
                          \end{array}
                        \right) \left(\begin{array}{cc}
                                 D_1 & cf_1^hf_2^vE \\
                                 O & I_{m_2}
                               \end{array}
\right).
\end{equation}
So
$$
det\left(
                  \begin{array}{cc}
                    D_1 & cf_1^hf_2^vE \\
                    cf_1^hf_2^v ~^tE & D_2\\
                  \end{array}
                \right)=(f_1^h)^{2m_2}(f_2^v)^{2m_1}det\left(I-c^2~^tEE\right).
$$
and we compute
$$
I-c^2~^tEE=-\left(
             \begin{array}{cccccc}
                \lambda d_1^2-1& \lambda d_1d_2 & \lambda d_1d_3&\cdots & \lambda d_1d_{m_2}  \\
               \lambda d_1d_2  & \lambda d_2^2-1 & \lambda d_2d_3 & \cdots  &  \lambda d_2d_{m_2} \\
               \vdots&\cdots & \ddots & \ldots &\ldots \\
                 \vdots&\cdots & \cdots & \ddots &\ldots  \\
                \lambda d_1d_{m_2} &  \lambda d_2d_{m_2} & \lambda d_3d_{m_2} & \ldots&  \lambda d_{m_2}^2-1 \\

             \end{array}
           \right).
$$
Where $\lambda=c^2\sum{\hskip -0.45cm_{_{_{_{_{_{i=1}}}}}}}{\hskip -0.5cm^{^{^{{m_1}}}}}(e_i(f_1)^h)^2$
and $d_j=e_{m_1+j}(f_2)^v$.\\
By a straightforward long computation using a limited recurrence gives
$$
(P_m)\left\{
  \begin{array}{ll}
    det\left(
             \begin{array}{cccccc}
                d\hskip-0.1cm_{_{_{11}}}\hskip-0.2cm-1&d_{12} & d\hskip-0.1cm_{_{_{13}}}&\cdots & d\hskip-0.1cm_{_{_{1m}}} \\
               d\hskip-0.1cm_{_{_{21}}}  & d\hskip-0.1cm_{_{_{22}}}\hskip-0.2cm-1& d\hskip-0.1cm_{_{_{23}}} & \cdots& d\hskip-0.1cm_{_{_{2m}}} \\
               \vdots&\cdots & \ddots & \ldots &\ldots \\
                 \vdots&\cdots & \cdots & \ddots &\ldots  \\
                d\hskip-0.1cm_{_{_{m1}}} & d\hskip-0.1cm_{_{_{m2}}} & d\hskip-0.1cm_{_{_{m3}}} & \ldots&  d\hskip-0.1cm_{_{_{mm}}}\hskip-0.4cm-1 \\

             \end{array}
           \right)=(-1)^{m}\left(1-\lambda\sum{\hskip -0.45cm_{_{_{_{_{_{j=1}}}}}}}{\hskip -0.5cm^{^{^{^{{m}}}}}}d_j^2\right), &  \\
&\\
    det\left(
             \begin{array}{ccccccccccccccc}
\hskip-0.2cmd\hskip-0.1cm_{_{_{i1}}}&\hskip-0.2cmd\hskip-0.1cm_{_{_{i2}}} &\hskip-0.2cmd\hskip-0.1cm_{_{_{i3}}}&\hskip-0.2cm\cdots &\hskip-0.2cm \cdots&\hskip-0.2cmd\hskip-0.1cm_{_{_{ii-1}}}&\hskip-0.2cmd\hskip-0.1cm_{_{_{ii+1}}}&\hskip-0.2cm\cdots&\hskip-0.2cm\cdots&
\hskip-0.2cmd\hskip-0.1cm_{_{_{im}}} \\
\hskip-0.2cmd\hskip-0.1cm_{_{_{21}}} &\hskip-0.2cm d\hskip-0.1cm_{_{_{22}}}\hskip-0.2cm-1 &\hskip-0.2cmd\hskip-0.1cm_{_{_{23}}}&\hskip-0.2cm \cdots  &\hskip-0.2cm \cdots&\hskip-0.2cmd\hskip-0.1cm_{_{_{2i-1}}}& \hskip-0.2cmd\hskip-0.1cm_{_{_{2i+1}}}&\hskip-0.2cm\cdots&\hskip-0.2cm\cdots&\hskip-0.2cmd\hskip-0.1cm_{_{_{2m}}} \\
\hskip-0.2cm\vdots & \hskip-0.2cm\cdots & \hskip-0.2cm\ddots & \hskip-0.2cm\cdots  &\hskip-0.2cm\cdots&\hskip-0.2cm\ldots& \hskip-0.2cm\ldots&\hskip-0.2cm\cdots&\hskip-0.2cm\cdots&\hskip-0.2cm\cdots \\
\hskip-0.2cm\vdots &\hskip-0.2cm \cdots &\hskip-0.2cm \cdots &\hskip-0.2cm \cdots  &\hskip-0.2cm\ddots&\hskip-0.2cm\ldots& \hskip-0.2cm\ldots&\hskip-0.2cm\cdots&\hskip-0.2cm\cdots&\hskip-0.2cm\cdots \\
\hskip-0.2cmd\hskip-0.1cm_{_{_{i-11}}}  &\hskip-0.2cmd\hskip-0.1cm_{_{_{i-12}}} &\hskip-0.2cmd\hskip-0.1cm_{_{_{i-13}}} &\hskip-0.2cm\cdots  & \hskip-0.2cmd\hskip-0.1cm_{_{_{i-1i-2}}}&\hskip-0.2cm\lambda d\hskip-0.1cm_{_{_{i-1i-1}}}\hskip-0.2cm-1 &\hskip-0.2cmd\hskip-0.1cm_{_{_{i-1i+1}}}&\hskip-0.2cm\cdots&\hskip-0.2cm\cdots&\hskip-0.2cmd\hskip-0.1cm_{_{_{i-1m}}} \\
\hskip-0.2cmd\hskip-0.1cm_{_{_{i+11}}}  &\hskip-0.2cmd\hskip-0.1cm_{_{_{i+12}}} &\hskip-0.2cmd\hskip-0.1cm_{_{_{i+13}}} &\hskip-0.2cm \cdots  &\hskip-0.2cmd\hskip-0.1cm_{_{_{i+1i-2}}}&\hskip-0.2cmd\hskip-0.1cm_{_{_{i+1i-1}}}&\hskip-0.2cm d\hskip-0.1cm_{_{_{i+1i+1}}}\hskip-0.2cm-1&\hskip-0.2cmd\hskip-0.1cm_{_{_{i+1i+2}}}&\hskip-0.2cm\cdots&\hskip-0.2cmd\hskip-0.1cm_{_{_{i+1m}}}\\
\hskip-0.2cm\vdots &\hskip-0.2cm \cdots &\hskip-0.2cm \cdots & \hskip-0.2cm\cdots  &\hskip-0.2cm\cdots&\hskip-0.2cm\ldots& \hskip-0.2cm\ldots&\hskip-0.2cm\ddots&\hskip-0.2cm\cdots&\hskip-0.2cm\cdots \\
\hskip-0.2cm\vdots & \hskip-0.2cm\cdots & \hskip-0.2cm\cdots & \hskip-0.2cm\cdots  &\hskip-0.2cm\cdots&\hskip-0.2cm\ldots& \hskip-0.2cm\ldots&\hskip-0.2cm\cdots&\hskip-0.2cm\ddots&\hskip-0.2cm\cdots \\
\hskip-0.2cmd\hskip-0.1cm_{_{_{m1}}}&\hskip-0.2cmd\hskip-0.1cm_{_{_{m2}}} &\hskip-0.2cmd\hskip-0.1cm_{_{_{m3}}}&\hskip-0.2cm\cdots &
\hskip-0.2cm\cdots&\hskip-0.2cmd\hskip-0.1cm_{_{_{mi-1}}}&\hskip-0.2cmd\hskip-0.1cm_{_{_{mi+1}}}&\hskip-0.2cm\cdots&\hskip-0.2cm\cdots
&\hskip-0.2cm\hskip -0.2cm\lambda d\hskip-0.1cm_{_{_{mm}}}\hskip-0.4cm-1 \\
 \end{array}
           \right)=(-1)^m\lambda d_1d_i. &
  \end{array}
\right.
$$
Where $d_{ij}=\lambda d_id_j$.\\
So,
  \begin{equation}\label{equ.degenerate}
\begin{array}{rl}
  det(\mathcal{M}_{f_1f_2})= &  det \left(
                  \begin{array}{cc}
                    D_1 & cf_1^hf_2^vE \\
                    cf_1^hf_2^v ~^tE & D_2\\
                  \end{array}
                \right) \\
  = & \! \!\!\!\!(f_1^h)\hskip-0.2cm^{^{^{2m_2}}}(f_2^v)\hskip-0.2cm^{^{^{2m_1}}}
\{1-c^2g_1(gradf_1,gradf_1)^hg_2(gradf_2,gradf_2)^v\},
\end{array}
\end{equation}
where $m_i$ $(i=1,2)$ is the dimension of $M_i$.
Since, $f_1$ and $f_2$ are non-constant smooth functions, then the proposition follows.
 \end{proof}
\begin{corollary}
If the symmetric tensor field $G_{f_1,f_2}$ of type $(0,2)$ on $M_1\times M_2$ is degenerate,
then for any  $i\in\{1,2\}$, $g_i(gradf_i,gradf_i)$ is positive constant $k_i$ with
$$
k_i=\frac{1}{c^2k_{(3-i)}}.
$$
\end{corollary}
\begin{proof}
Note that if  $G_{f_1,f_2}$ is degenerate then $c$ is non-zero real number,
$f_1,f_2$ is nonconstant smooth functions on $M_1$, $M_2$ respectively and we have
$$
c^2g_1(gradf_1,gradf_1)^hg_2(gradf_2,gradf_2)^v=1.
$$
Since $g_i(gradf_i,gradf_i)$ depend only on $M_i$, $(i=1,2)$ we conclude that
 $g_i(gradf_i,gradf_i)$ is constant.
\end{proof}
\begin{remark}
Under the same assumptions as in Proposition \ref{condition generalized warped}, if
$f_1,f_2$ are non-constant smooth functions on $M_1$, $M_2$
respectively and $\varphi$ is smooth function on $M_1\times M_2$ that satisfies
$\frac{-1}{\|gradf_1\|^h\|gradf_2\|^v}<\varphi<\frac{1}{\|gradf_1\|^h\|gradf_2\|^v}$, then
the symmetric tensor fields
$$
 G_{_{f_1,f_2}}=(f_2^v)^2\pi_1^*g_{_{_1}}+(f_1^h)^2\pi_2^*g_{_{_2}}
        +\varphi f_1^hf_2^v df_1^h\odot df_2^v.
$$
 is Riemannian metric on $M_1\times M_2$.
\end{remark}

In all what follows, we suppose that $f_1$ and $f_2$ satisfies the inequality (\ref{condition of metric}).
\begin{lemma}\label{calculate of X on G}
 Let $X$ be an arbitrary vector field of $M_1\times M_2$, if there exist $\varphi_i,\psi_i\in C^{\infty}(M_i)$
and $X_i,Y_i\in \Gamma(TM_i)$, $(i=1,2)$ such that
$$
  \left\{
     \begin{array}{lll}
       G_{f_1f_2}(X,Z_1^h)= G_{f_1f_2}(\varphi_2^vX_1^h+\varphi_1^hX_2^v,Z_1^h), & \\
&&\forall ~Z_i\in \Gamma(TM_i),\\
       G_{f_1f_2}(X,Z_2^v)= h^hG_{f_1f_2}(\psi_2^vY_1^h+\psi_1^hY_2^v,Z_2^v). &
     \end{array}
   \right.
$$
Then we have,
\begin{equation}
\begin{array}{ccl}
  X& = &\varphi_2^vX_1^h+\psi_1^hY_2^v+cf_1^hf_2^v\left\{\psi_{2}^vY_1(f_1)^h
\!-\!\varphi_{2}^vX_1(f_1)^h\right\}grad(f_{2}^v) \\
&&\\
 &-&cf_1^hf_2^v\left\{\psi_{1}^hY_2(f_2)^v
\!-\!\varphi_{1}^hX_2(f_2)^v\right\}grad(f_{1}^h)
\end{array}
\end{equation}
\end{lemma}
\begin{proof}At first, we put
$$
B=X-\varphi_2^vX_1^h-\psi_1^hY_2^v \hskip 0.3cm
 \text{and}~\hskip 0.3cm Z=Z_1^h+Z_2^v.
$$
It suffices to observe that
$$
\frac{-1}{c f_1^hf_2^v}G_{f_1f_2}(B,Z) =\frac{1}{c f_1^hf_2^v}\left\{G_{f_1f_2}(\psi_1^hY_2^v-\varphi_1^hX_2^v,Z_1^h)+
G_{f_1f_2}(\varphi_2^vX_1^h-\psi_2^vY_1^h,Z_2^v) \right\}
$$
$$
\hskip 2cm= \left\{ (\psi_1^hY_2^v(f_2^v)-\varphi_1^hX_2^v(f_2^v))Z_1^h(f_1^h)
+(\varphi_2^vX_1^h(f_1^h)-\psi_2^vY_1^h(f_1^h))Z_2^v(f_2^v)\right\}
$$
$$
\hskip 0.9cm =\sum_{i=1}^2(-1){^{^{^{i}}}}G_{f_1f_2}(\left\{\psi_{3-i}^JY_i(f_i)^I
-\varphi_{3-i}^JX_i(f_i)^I\right\}grad(f_{3-i}^J) ,Z).
$$
With $(i,I),(3-i,J)\in \{(1,h),(2,v)\}$. The result follows.
\end{proof}

\subsection{The Levi-Civita Connection}
\begin{lemma}\label{grad}
Let $(M_i,g_i)$, $(i=1,2)$ be a Riemannian manifold. The gradient of the lifts $f_1^h$ of $f_{1}$
 and $f_2^v$ of $f_{2}$ to $M_1\times_{f_1,f_2}M_2$  w.r.t. $G_{f_1,f_2}$ are
\begin{equation}
 grad(f_1^h)\!=\!\frac{1}{1-c^2b_1^hb_2^v}
\big\{\frac{1}{(f_{2}^v)^2}(gradf_1)^h\!-\!\frac{cb_1^h}{f_1^hf_2^v}
(gradf_{2})^v\big\},
\end{equation}
\begin{equation}
 grad(f_2^v)\!=\!\frac{1}{1-c^2b_1^hb_2^v}
\big\{\frac{1}{(f_{1}^h)^2}(gradf_2)^v\!-\!\frac{cb_2^v}{f_1^hf_2^v}
(gradf_{1})^h\big\},
\end{equation}
where $b_i=\| gradf_i\|^2$ (i=1,2).
\end{lemma}
\begin{proof}
           Let $Z_i\in \Gamma(TM_{_i})$, $i=1,2$, then for $(i,I),(3-i,J)\in \{(1,h),(2,v)\}$, we have,
$$G_{_{f_1f_2}}(grad(f_i^I),Z_i^I) =  \frac{1}{(f_{3-i}^J)^2}G_{_{f_1f_2}}((gradf_i)^I,Z_i^I),$$
and
$$ G_{_{f_1f_2}}(grad(f_i^I),Z_{3-i}^J)=0.$$
Therefor, the result follows by Equation (\ref{equivalent generalized warped}) and Lemma \ref{calculate of X on G}.

\end{proof}
\begin{lemma}
Let $(M_{_i},g_{i})$, $(i=1,2)$ be a Riemannian manifold and let $\varphi_{i}$ be a
smooth function on $M_{_i}$. The gradient of the lifts $\varphi{_1}^h$ of $\varphi_{1}$
and $\varphi{_2}^v$ of $\varphi_{2}$ to $M_{_1}\times_{f_1,f_2}M_{_2}$
 w.r.t. $G_{_{f_1f_2}}$ are
 \begin{equation}
grad(\varphi_{1}^h)=\left(\frac{1}{f_{2}^v} \right)^2\!\!\!\left(grad\varphi_1\right)^h
-c\frac{\left(f_1grad\varphi_1(f_1)\right)^h}{f_{2}^v}grad(f_{2}^v),
 \end{equation}

 \begin{equation}
grad(\varphi_{2}^v)=\left(\frac{1}{f_{1}^h} \right)^2\!\!\!\left(grad\varphi_2\right)^v
-c\frac{\left(f_2grad\varphi_2(f_2)\right)^v}{f_{1}^h}grad(f_{1}^h),
\end{equation}
\end{lemma}
\begin{proof}
  Let $Z_i\in \Gamma(TM_{_i})$, $(i=1,2)$ then for $(i,I),(3-i,J)\in \{(1,h),(2,v)\}$, we have,
$$
G_{_{f_1f_2}}(grad(\varphi_i^I),Z_i^I) =  \frac{1}{(f_{3-i}^J)^2}G_{_{f_1f_2}}((grad\varphi_i)^h,Z_i^I),
$$
and
$$ G_{_{f_1f_2}}(grad(\varphi_i^I),Z_{3-i}^J)=0.$$
Therefor, the result follows by Equation (\ref{equivalent generalized warped})and Lemma \ref{calculate of X on G}.
\end{proof}
\begin{proposition}
Let $(M_{_1},g_i)$, $(i=1,2)$ be a pseudo-Riemannian manifold and let
$f_i: M_{_i}\rightarrow \mathbb{R}_+^*$, be a positive smooth function.
 The cometric $\widetilde{G}{\hskip -0.11cm_{_{f_1f_2}}}$ \!\!of $G{\hskip -0.11cm_{_{f_1f_2}}}$ is given by
\begin{equation}
\begin{array}{ll}
\widetilde{G}_{_{f_1f_2}}&=\left(\frac{1}{f_2^v}\right)^2\widetilde{g}_{_{_1}}^h+\left(\frac{1}{f_1^h}\right)^2\widetilde{g}_{_{_2}}^v
+\frac{1}{1-c^2b_1^hb_2^v}\big\{ \frac{c^2b_2^v}{(f_2^v)^2}(gradf_1)^h\odot(gradf_1)^h\\
&+\frac{c^2b_1^h}{(f_1^h)^2}(gradf_2)^v\odot(gradf_2)^v-\frac{c}{f_{_1}^hf_{_2}^v}(gradf_1)^h\odot(gradf_2)^v\big\}.
\end{array}
\end{equation}
It is the unique tensor fields such that
\begin{equation} \label{cometric of generalized warped}
   \widetilde{G}\hskip -0.1cm_{_{_{_{f_1f_2}}}}\hskip -0.2cm(\alpha_i^I,\beta_k^K)\!= \!\left\{
   \begin{array}{lll}
\!\!
\!\frac{1}{(f_{j}^J)^2}\!\left\{\!\widetilde{g}_{i}(\alpha_i,\beta_i)^I
\!+\!\frac{c^2b_j^J}{1-c^2b_1^hb_2^v}
\widetilde{g}_{i}(\alpha_i,df_i)^I\widetilde{g}_{i}(\beta_i,df_i)^I\!\!\right\}, if~i=k &\\
&&\\
\!\frac{-c}{f_1^hf_2^v(1-c^2b_1^hb_2^v)}
\widetilde{g}_{i}(\alpha_i,df_i)^I\widetilde{g}_k(\beta_{k},df_{k})^K. \hskip 1.5cm if ~i\neq k&\\

\end{array}
     \right.
\end{equation}
for any $\alpha_i,\beta_i\in \Gamma(T^*M_{_i})$  $(i=1,2 ~\text{and}~j=3-i)$.
Where $\widetilde{g_i}$ $(i=1,2)$ is the cometrics of $g_{_{_i}}$ and $(i,I),(k,K),(j,J)\in \{(1,h),(2,v)\}$.

   \end{proposition}
\begin{proof}
 A direct computation using Equation \ref{equivalent generalized warped}, the definition of the
musical isomorphismes and
$$
\sharp{_{_{G\hskip-0.1cm_{_{_{_{f_1f_2}}}}}}}\hskip-0.5cm(\alpha_i^I)=
\left(\frac{1}{f_{3-i}^J}\right)^2\left(\sharp_{_{_{_{_{g_i}}}}}\hskip -0.2cm(\alpha_i)\right)^I
-\frac{cf_i^I}{f_{3-i}^J}\widetilde{g}_i(\alpha_i,df_i)^hgrad(f_{3-i}^J),
$$
for $(i,I),(3-i,J)\in \{(1,h),(2,v)\}$, leads to gives
 (\ref{cometric of generalized warped}).
\end{proof}
let us compute the Levi-Civita connection of $M_1\times_{f_1f_2}M_2$ associated with
 the metric $G_{f_1f_2}$ in terms of the Levi-Civita connections $\nabla{\hskip -0.2cm^{^{^{1}}}}$
 and $\nabla{\hskip -0.2cm^{^{^{2}}}}$ associated with the metrics $g_1$ and $g_2$
respectively.
\begin{proposition}\label{generalized}
Let $(M_i,g_{i})$, $(i=1,2)$ be a Riemannian manifold. Then we have

\begin{equation}\label{Horisontal }
  \nabla _{X_1^h}Y_1^{h}=(\nabla{\hskip -0.2cm^{^{^{1}}}}_{X_1}Y_1)^h+f_{_{2}}^v
B{\hskip -0.1cm_{_{f\!_{_{1}}}}}\!\!(X_1,Y_1)^hgrad(f_{_{2}}^v)
\end{equation}
\begin{equation}\label{ vertical}
  \nabla _{X_2^v}Y_2^{v}=(\nabla{\hskip -0.2cm^{^{^{2}}}}_{X_2}Y_2)^v+f_{_{1}}^h
B{\hskip -0.1cm_{_{f\!_{_{2}}}}}\!\!(X_2,Y_2)^hgrad(f_{_{1}}^h)
\end{equation}
\begin{equation}\label{milange}
 \begin{array}{rl}
 \hskip 0.3cm\nabla _{X_1^h}Y_2^{v}= \nabla _{Y_2^v}X_1^{h}&=-cX_1(f_{_{1}})^hY_{2}(f_{_{2}})^v\big\{f_{_2}^vgrad(f_{_{1}}^h)
  +f_{_{1}}^hgrad(f_{_{2}}^v)\}\\
&+\big(Y_2(\ln f_{_{2}})\big)^vX_1^h+\big(X_1(\ln f_{_{1}})\big)^hY_2^v,
 \end{array}
 \end{equation}
Where $B{\hskip -0.1cm_{_{f\!_{_{i}}}}}$, $(i=1,2)$ the symmetric
 $(0,2)$ tensor field of $f_{_i}$ given by
$$
B{\hskip -0.1cm_{_{f\!_{_{i}}}}}(X_i,Y_i)=cf_{_i}H^{_{f_{_i}}}(X_i,Y_i)+cX_i(f_{_i})Y_i(f_{_i})-g_i(X_i,Y_i),
$$
$H^{f_i}$ is the Hessian of $f_i$.
\end{proposition}
\begin{proof}
 Let $X_i,Y_i, Z_i\!\!\in \!\Gamma(TM_i)$, $i=1,2$. For any
$(i,I), (k,K)\!\in \!\!\{(1,h),(2,v)\}$ we have
\begin{equation}\label{koszulijk}
\begin{array}{ll}
2{G_{f}}(\nabla_{X_i^I}\!\!\!\!\!&Y_i^I,Z_k^K) \!=
\!X_i^I({G_{f_1f_2}}(Y_i^I,Z_k^K))+Y_i^I({G_{f_1f_2}}(X_i^I,
Z_k^K))\!-\!Z_k^K({G_{f_1f_2}}(X_i^I,Y_i^I))\\
&\!\!\!\!+{G_{f_1f_2}}([X_i^I,Y_i^I],Z_k^K)+{G_{f_1f_2}}([Z_k^K,X_i^I],Y_i^I)
+{G_{f_1f_2}}([Z_k^K,Y_i^I], X_i^I).
\end{array}
\end{equation}
1. Taking $(i,I)\!=\!(k,K)$ in this formula, using Formula (\ref{equivalent generalized warped})
and Lemma \ref{lift}, we get
$$
2G_{f_1f_2}(\nabla_{X_i^h}Y_i^I,Z_i^I)=2(f_{3-i}^J)^2(g_i(\nabla{\hskip -0.2cm^{^{^{i}}}}_{X_i}Y_i,Z_i))^I,
$$
and using (\ref{equivalent generalized warped}) again, we get
$$
G_{f_1f_2}(\nabla_{X_i^I}Y_i^I,Z_i^I)=G_{f_1f_2}((\nabla{\hskip -0.2cm^{^{^{i}}}}_{X_i}Y_i)^I,Z_i^I).
$$
Similarly, taking $(i,I)\neq(k,K)$, we get
$$
{G_{f_1f_2}}(\nabla_{X_i^I}Y_i^I,Z_k^K)=
\left(\frac{cX_i(f_iY_i(f_i)-g_i(X_i,Y_i))}{(f_i^2)}\right)^IG_{f_1f_2}\big((f_kgradf_k)^K,Z_k^K\big)
$$
The result then follows by Lemma \ref{calculate of X on G}. \smallskip\\
2. Taking $i\neq k$.\\
 At first, since $\nabla$ is
torsion-free we have $\nabla_{Y_k^K}X_i^I=\nabla_{X_i^I}Y_k^K+[X_i^I,Y_k^K]$.
By Lemma \ref{lift}, we have $[X_i^I,Y_k^K]=0$. This implies that
$\nabla _{X_i^I}Y_k^{K}=\nabla_{Y_k^K}X_i^I$.\\
Using Formula (\ref{equivalent generalized warped}) and Lemma \ref{lift}, we get
\begin{align*}
G_{f_1f_2}(\nabla_{X_i^I}Y_j^K,Z_i^I)=G_{f_1f_2}\big((\frac{Y_k(f_k)}{f_k})^KX_i^I,Z_i^I\big),
\end{align*}
and
\begin{align*}
G_{f_1f_2}(\nabla_{X_i^I}Y_k^K,Z_k^K)=G_{f_1f_2}\big((\frac{X_i(f_i)}{f_i})^IY_k^K,Z_k^K\big).
\end{align*}
Thus the result follows by Lemma \ref{calculate of X on G}.

\end{proof}

\subsection{The Laplacian of the lifts to $M_1$ and $M_2$}
\begin{theorem}\label{Laplacian on generalized}
On a generalized warped product $(M_{_1}\times_{f_{_1}f_{_2}} M_{_2}, G_{f_{_1}f_{_2}})$ with $m_{_1}=dim M_{_1}$
and $m_{_2}=dim M_{_2}$, Let $f_{_1}:M_1 \rightarrow \mathbb{R}$
and $f_{_2}:M_1 \rightarrow \mathbb{R}$ be a smooth functions. Then the Laplacian
of the horizontal lift $f_{_1}\circ \pi_1$ of $f_{_1}$  (resp. vertical lift $f_{_2}\circ \pi_2$ of $f_{_2}$ )
to $M_{_1}\times_{f_1f_2} M_{_2}$ is given by
\begin{equation}
\Delta(f_1^h)=\frac{1}{f_2^v(1-c^2b_1^hb_2^v)}\left\{\frac{1}{f_2^v}\big(\Delta_1(f_1)\big)^h-\frac{cb_1^h}{f_1^h}\big(\Delta_2(f_2)\big)^v
+\frac{b_1^h\left(c(1-m_1)b_2^v+m_2\right)}{f_1^hf_2^v}\right\}
\end{equation}
$$
+\frac{c^2}{2f_2^v(1-c^2b_1^hb_2^v)^2}\left\{\frac{b_2^v}{f_2^v}\left(gradf_1(b_1)\right)^h
-\frac{c(b_1^2)^h}{f_1^h}\left(gradf_2(b_2)\right)^v\right\}.
$$
\begin{equation}
\Delta(f_2^v)=\frac{1}{f_1^h(1-c^2b_1^hb_2^v)}\left\{\frac{1}{f_1^h}\big(\Delta_2(f_2)\big)^v-\frac{cb_2^v}{f_2^v}\big(\Delta_1(f_1)\big)^h
+\frac{b_2^v\left(c(1-m_2)b_1^h+m_1\right)}{f_1^hf_2^v}\right\}
\end{equation}
$$
+\frac{c^2}{2f_1^h(1-c^2b_1^hb_2^v)^2}\left\{\frac{b_1^h}{f_1^h}\left(gradf_2(b_2)\right)^v
-\frac{c(b_2^2)^v}{f_2^v}\left(gradf_1(b_1)\right)^h\right\}.
$$
Where $b_i=\| gradf_i\|^2$ (i=1,2).
\end{theorem}
\begin{lemma}\label{orthonormal basis}
On $(M_{_1}\times_{f_{_1}f_{_2}} M_{_2}, G_{f_{_1}f_{_2}})$, if $\{e_{_1},...,e_{_{m_{_1}}}\}$ is the local frame field
with respect to the metric $g_{_1}$ and $\{e_{_{m_1+1}},...,e_{_{m_1+m_2}}\}$ is the local frame field
with respect to the metric $g_{_2}$, then $\{u_{_1},...,u_{_{m_1}},u_{_{m_1+1}},...,u_{_{m_1+m_2}}\}$ is the local frame field
with repect to the metric $G_{f_{_1}f_{_2}}$, where
\begin{equation}
  u'_j\!=\!\left\{
     \begin{array}{ll}
\!\!\!\!\frac{1}{f_{_2}^v}e_j^h, &j\!\in\!\{1,...,m_{_1}\}; \\
\!\!\!\!\frac{ca_j^v}{(1-c^2b_1^hA_j^v)}\!\left\{-\frac{1}{f_2^v}(gradf_1)^h
\!+\!\frac{cb_1^h}{f_1^h}T_j^v\right\}\!+\!\frac{1}{f_1^h}e_j^v,& \!\!\!\!j\!\!\in\!\{m_1+1,.,m_1+m_2\}.
     \end{array}
   \right.
\end{equation}
And for $j\in\{m_1+1,...,m_1+m_2\}$,
$$
u_j\!\!=\!\!\frac{1}{\|u'_j\|}u'_j, \hskip 0.1cm \|u_j'\|^2\!=\!\frac{1-c^2b_1^hA_{j+1}^v}{1-c^2b_1^hA_{j}^v},
\hskip 0.1cm A_j\!=\!\!\!\!\!\!\sum_{i=m_{_1}+1}^{j-1}\!\!\!\!\!a_{_i}^2,
 \hskip 0.12cm T_j\!=\!\!\!\!\!\!\!\sum_{i=m_{_1}+1}^{j-1}\!\!\!\!\!a_ie_i,\hskip 0.1cm
a_i=e_i(f_{_2}).
$$
\end{lemma}
\begin{proof}
We know that $ G_{f_{_1}f_{_2}}$ is Riemannian metric if and only if $0<1-b_{_1}^hb_{_2}^v$. Then if we choose
$\{e_{_1},...,e_{_{m_{_1}}}\}$ to be a local, orthonormal basis of the vector fields with respect to $g_{_1}$
on an open $O_{_1}\subset M_{_1}$  and $\{e_{_{m_1+1}},...,e_{_{m_1+m_2}}\}$ to be a local orthonormal basis of
the vector fields with respect to the metric $g_{_2}$ on an open $O_{_2}\subset M_{_2}$, then the family
$$
\{v_{_1}=\frac{1}{f_{_2}^v}e_{_1}^h,...,v_{_{m_1}}=\frac{1}{f_{_2}^v}e_{_{m_1}}^h,
v_{_{m_1}+1}=\frac{1}{f_{_1}^h}e_{_{m_1}+1}^v,...,v_{_{m_1+m_2}}=\frac{1}{f_{_1}^h}e_{_{m_1+m_2}}^v\}
$$
is a local basis of the vector fields with respect to $G_{f_{_1}f_{_2}}$ on an open $O_{_1}\times O_{_2}\subset M_{_1}\times M_{_2}$.\\
The gradient of $f_{_1}$ (resp. $f_{_2}$ ) and its
norm $\|gradf_{_1}\|$ (resp. $\|gradf_{_2}\|$) can be written as

\begin{equation}\label{gradf_1}
gradf_{_1}=\sum_{k=1}^{m_1}e_k(f_1)e_k, \hskip 0.5cm \|gradf_{_1}\|^2=\sum_{k=1}^{m_1}(e_k(f_1))^2
\end{equation}
\begin{equation}\label{gradf_2}
\big(rep.\hskip0.2cm gradf_{_2}=\sum_{i=m_1+1}^{m_1+m_2}a_ie_i,\hskip 0.5cm \|gradf_{_2}\|^2=\sum_{i=m_1+1}^{m_1+m_2}a_i^2\big).
\end{equation}

$G_{f_{_1}f_{_2}}$ is positive definite, which implies that
\begin{equation}
1-c^2b_{_2}^h\sum_{k=1}^l(a_{_k}^h)^2>0, \hskip0.3cm \forall l\in\{1,...,m_{_1}\},
\end{equation}
and
\begin{equation}
 1-c^2b_{_1}^h\sum_{i=m_{_1}+1}^j(a_{_i}^v)^2>0, \hskip 0.3cm \forall j\in\{m_{_1}+1,...,m_{_1}+m_{_2}\}.
\end{equation}
For the proof of the lemma it is actually almost the most interesting result because it provides
an algorithm for constructing $\{u_{_1},...,u_{_{m_1}},u_{_{m_1+1}},...,u_{_{m_1+m_2}}\}$
from the family $\{e_{_1},...,e_{_{m_{_1}}}\}$ et $\{e_{_{m_1+1}},...,e_{_{m_1+m_2}}\}$.\\
To do so, we use a limited recurrence (The Gram schmidt process).\\
At first, we put $u'_1=v_1$  and $u_{_1}=\frac{v_{_1}}{\|v_{_1}\|}$. For $j\in\{2,...,m_1,m_1+1,...,m_1+m_2\}$,
\begin{equation}\label{Gram schmidt process}
u'_j=v_j-\sum_{i=1}^{j-1}G_{f_{_1}f_{_2}}(v_j,u_i)u_i \hskip 0.4cm and \hskip 0.4 cm u_j=\frac{u'_j}{\|u'_j\|}.
\end{equation}
By virtue of (\ref{Gram schmidt process}), a straightforward calculation using (\ref{gradf_1}) and (\ref{gradf_2}) gives
$$
u_k=\frac{1}{f_{_2}^v}e_k^h, ~\hskip 0.3 cm \|u_k\|=1~ \hskip 0.2cm \forall k\in \{1,...,m_1\},
$$
for all $j\in \{m_1,...,m_1+m_2\}$, we have
$$
\begin{array}{rcl}
  u'_j &=&\frac{-ca_j^v}{f_{_2}^v\left(1-c^2b_1^h\sum\hskip-0.4cm_{_{_{_{_{i=m_1+1}}}}}
^{^{^{^{^{j-1}}}}}\hskip -0.5cm(a_i^v)^2\right)}(gradf_{_1})^h
+\frac{1}{f_{_1}^h}e_j^v\\
&+&\frac{c^2b_1^ha_j^v}{f_{_1}^h\left(1-c^2b_1^h\sum\hskip-0.4cm_{_{_{_{_{i=m_1+1}}}}}
^{^{^{^{^{j-1}}}}}\hskip -0.5cm(a_i^v)^2\right)}
\left(\sum\hskip-0.4cm_{_{_{_{_{i=m_1+1}}}}}
^{^{^{^{j-1}}}}\hskip -0.4cm a_ie_i\right)^v,
\end{array}
$$
and
$$
\|u'_j\|=\sqrt{\frac{\left(1-c^2b_1^h\sum\hskip-0.4cm_{_{_{_{_{i=m_1+1}}}}}
^{^{^{^{^{j}}}}}\hskip -0.5cm(a_i^v)^2\right)}{\left(1-c^2b_1^h\sum\hskip-0.4cm_{_{_{_{_{i=m_1+1}}}}}
^{^{^{^{^{j-1}}}}}\hskip -0.5cm(a_i^v)^2\right)}}.
$$
\end{proof}
\begin{remark}
With the notations above, we have \\
1) $T_{m_1+1}$ is the zero vector field on $M_{_2}$, $A_{m_1+1}$ is the zero function on $M_{_2}$ and
$A_{m_1+m_2}$ is the care of the gradient of $f_{_2}$.\\
2)For any $j\in\{m_1+1,...,m_1+m_2\}$
\begin{equation}\label{u_j}
\left\{
  \begin{array}{lll}
  T_j(f_{_2})=A_j=g_2(T_j,T_j), &  \\
   u'_j(f_{_1}^h)=-\frac{cb_1^h}{f_{_2}^v}\big(\frac{a_j^v}{1-c^2b_1^hA_j^v}\big)=-\frac{cf_{_1}^hb_1^h}{f_{_2}^v}u'_j(f_{_2}^v) , & \\
 u_j(f_{_1}^h)=-\frac{cf_{_1}^hb_1^h}{f_{_2}^v}u_j(f_{_2}^v) , &
  \end{array}
\right.
\end{equation}
\end{remark}
\begin{lemma}\label{X_j and e_j}
 With the notations above, we have

\begin{equation}\label{X_j^v+e_j^v}
\frac{1}{1-c^2b_1^hb_2^v}(gradf_{_2})^v
 =c^2b_1^h\hskip -0.3cm\sum_{j=m_1+1}^{m_1+m_2}\hskip -0.2cm\left(\frac{a_j^v}{\|u'_j\|(1-c^2b_1^hA_j^v)}\right)^2T_j^v
+\hskip -0.3cm\sum_{j=m_1+1}^{m_1+m_2}\hskip -0.2cm\left(\frac{a_j^v}{\|u'_j\|^2 (1-c^2b_1^hA_j^v)}\right)e_j^v.
\end{equation}
And
\begin{equation}
  \frac{b_2^v}{1-c^2b_1^hb_2^v}
 =c^2b_1^h\hskip -0.3cm\sum_{j=m_1+1}^{m_1+m_2}\hskip -0.2cm\left(\frac{a_j^v\sqrt{A_j^v}}{\|u'_j\| (1-c^2b_1^hA_j^v)}\right){\!\!\!^{^{^{^{^{^{2}}}}}}}
\hskip -0.1cm+\hskip -0.3cm \sum_{j=m_1+1}^{m_1+m_2}\hskip -0.2cm\left(\frac{(a_j^v)^2}{\|u'_j\|^2(1-c^2b_1^hA_j^v)}\right).
\end{equation}

\end{lemma}
\begin{proof}
From Lemma \ref{orthonormal basis}, $\{u_{_1},...,u_{_{m_1}},u_{_{m_1+1}},...,u_{_{m_1+m_2}}\}$ is the local frame field
with repect to the metric $G_{f_{_1}f_{_2}}$, then
\begin{equation*}\label{grad(f_1^h)}
grad(f_1^h)=\sum_{j=1}^{m_1+m_2}u_j(f_1^h)u_j=\big(\frac{1}{f_2^v}\big)^2
\sum_{j=1}^{m_1}e_j^h(f_1^h)e_j^h+\sum_{j=m_1+1}^{m_1+m_2}u_j(f_1^h)u_j
\end{equation*}
$$
\hskip2cm=\big(\frac{1}{f_2^v}\big)^2(gradf_1)^h+\frac{c^2b_1^h}{(f_2^v)^2}
\sum_{j=m_1+1}^{m_1+m_2}\big(\frac{a_j}{\|u'_j\|(1-c^2b_1^hA_j^v)}\big)^2(gradf_1)^h
$$
$$
\hskip2.5cm-\frac{c^3(b_1^2)^h}{f_1^hf_2^v}\sum_{j=m_1+1}^{m_1+m_2}\big(\frac{a_j}{\|u'_j\|(1-c^2b_1^hA_j^v)}\big)^2T_j^v
-\frac{cb_1^h}{f_1^hf_2^v}\sum_{j=m_1+1}^{m_1+m_2}\frac{a_j}{\|u'_j\|^2(1-c^2b_1^hA_j^v)}e_j^v.
$$
On the other hand, by (\ref{u_j}), we have also,
$$
\big(\frac{cb_1^h}{f_2^v}\big)^2\sum_{j=m_1+1}^{m_1+m_2}\big(\frac{a_j}{\|u'_j\|(1-c^2b_1^hA_j^v)}\big)^2=\sum_{j=1}^{m_1+m_2}\big(u_j(f_1^h)\big)^2
-\big(\frac{1}{f_2^v}\big)\big(\sum_{k=1}^{m_1}\big(e_k(f_1)\big)^2\big)^h
$$
$$
\hskip3.9cm=\|grad(f_1^h)\|^2-\left(\frac{\|gradf_1\|^h}{f_2^v}\right)^2
$$
$$
\hskip 2cm=\frac{c^2(b_1^2)^hb_2^v}{(f_2^v)^2(1-c^2b_1^hb_2^v)}.
$$
Substituting in the previous equation then leads to the required result.\\
The second assertion can be calculated by applying Equation (\ref{X_j^v+e_j^v}) the function $f_2^v$.
\end{proof}
\begin{lemma}\label{sum 1/B_j and B_j}
  With the notations above, we have, for all $j\in\{m_1+1,...,m_1+m_2\}$
\begin{equation}\label{sum 1/B_j}
\frac{1}{1-c^2b_1^hA_j^v}+c^2b_1^h\sum_{i=j}^{m_1+m_2}\frac{(a_i^v)^2}{(1-c^2b_1^hA_i^v)(1-b_1^hA_{i+1}^v)}=\frac{1}{1-c^2b_1^hb_2^v},
\end{equation}
\begin{equation}\label{sumB_j}
\frac{(1-c^2b_1^hA_{j+1}^v)(1-c^2b_1^hA_{j-1}^v)+(c^2b_1^ha_j^va_{j-1}^v)^2}{(1-c^2b_1^hA_{j}^v)
\big(1-c^2b_1^h(A_{j+1}^v-(a_{j-1}^v)^2)\big)}=1,
\end{equation}
and
\begin{equation}
\frac{1}{\|u'_j\|^2}+(c^2b_1^ha_j^v)^2\!\!\!\!\sum_{i=j+1}^{m_1+m_2}\!\!\!\frac{(a_i^v)^2}{(1-c^2b_1^hA_i^v)(1-c^2b_1^hA_{i+1}^v)}
=\frac{1-c^2b_1^h(b_2^v-(a_j^v)^2)}{1-c^2b_1^hb_2^v}.
\end{equation}
 \end{lemma}
\begin{proof}
Firstly, if we put $B_j=1-c^2b_1^hA_j$ and $C^{i,i+1}=B_j...B_{i-1}B_{i+2}...B_{m_1+m_2+1}$ with $i\geq j$, then
for any $j\in\{m_1+1,...,m_1+m_2\}$, the equation
$\frac{1}{B_j}+c^2b_1^h\sum\hskip -0.5cm_{_{_{_{_{i=j}}}}}^{^{^{^{m_1+m_2}}}}\hskip-0.2cm\frac{(a_i^v)^2}{B_i B_{i+1}}$ becomes
$$
\frac{\prod\hskip -0.5cm _{_{_{_{_{_{k=j+1}}}}}} \hskip -1cm^{^{^{^{m_1+m_2+1}}}}\hskip-0.5cm B_{k}
+c^2b_1^h\sum\hskip -0.5cm _{_{_{_{_{_{i=j}}}}}} \hskip -0.7cm^{^{^{^{^{m_1+m_2}}}}}\hskip-.2cm(a_i^v)^2C^{i,i+1}}
{\prod\hskip -0.5cm _{_{_{_{_{_{k=j}}}}}} \hskip -0.8cm^{^{^{^{^{m_1+m_2+1}}}}}\hskip-0.5cmB_{k}},
$$
furthermore,
$$
\prod\hskip -0.5cm _{_{_{_{_{_{_{_{k=j+1}}}}}}}} \hskip -1cm^{^{^{^{^{m_1+m_2+1}}}}}\hskip-0.5cm B_{k}
+c^2b_1^h\sum\hskip -0.5cm _{_{_{_{_{_{_{_{i=j}}}}}}}} \hskip -0.7cm^{^{^{^{^{^{m_1+m_2}}}}}}\hskip-.2cm(a_i^v)^2C^{i,i+1}
=C^{j,j+1}\big(B_{j+1}+c^2b_1^h(a_j^v)^2\big)
+c^2b_1^h\sum\hskip -0.7cm _{_{_{_{_{_{_{_{i=j+1}}}}}}}} \hskip -0.8cm^{^{^{^{^{^{m_1+m_2}}}}}}\hskip-.2cm(a_i^v)^2C^{i,i+1}.
$$
Therefore, using the fact that $B_jC^{j,j+1}=B_{j+2}C^{j+1,j+2}$, $B_{j+1}+c^2b_1^h(a_j^v)^2= B_j$ and by induction, we have
$$
\prod\hskip -0.5cm _{_{_{_{_{_{_{_{_{k=j+1}}}}}}}}} \hskip -1cm^{^{^{^{^{m_1+m_2+1}}}}}\hskip-0.5cm B_{k}
+c^2b_1^h\sum\hskip -0.5cm _{_{_{_{_{_{_{i=j}}}}}}} \hskip -0.7cm^{^{^{^{^{m_1+m_2}}}}}\hskip-.2cm(a_i^v)^2C^{i,i+1}
=\prod\hskip -0.5cm _{_{_{_{_{_{_{_{k=j+1}}}}}}}} \hskip -1cm^{^{^{^{^{^{m_1+m_2-1}}}}}}\hskip-0.5cm B_{k}\big(B_{m_1+m_2+1}+c^2b_1^h(a_{m_1+m_2}^v)^2\big).
$$
Accordingly,
$$
\frac{1}{B_j}+c^2b_1^h\hskip-.2cm\sum_{i=j}^{m_1+m_2}\hskip-0.2cm\frac{(a_i^v)^2}{B_{i+1}B_i}=
\frac{\prod\hskip -0.5cm _{_{_{_{_{_{_{_{k=j+1}}}}}}}} \hskip -1cm^{^{^{^{^{^{m_1+m_2-1}}}}}}\hskip-0.5cm
 B_{k}\big(B_{m_1+m_2+1}+c^2b_1^h(a_{m_1+m_2}^v)^2\big)}
{\prod\hskip -0.5cm _{_{_{_{_{_{k=j}}}}}} \hskip -0.8cm^{^{^{^{^{m_1+m_2+1}}}}}\hskip-0.5cmB_{k}}=\frac{1}{1-c^2b_1^hb_2^h}.
$$
For $j\in\{m_1+2,...,m_1+m_2\}$, we have
$$
\begin{array}{ccl}
  B_{j+1}B_{j-1}+ c^4(b_1^2)^h\big(a_j^va_{j-1}^v\big)^2 & = & B_{j+1}\big(B_j+c^2b_1^h(a_{j-1}^v)^2\big)+c^4(b_1^2)^h\big(a_j^va_{j-1}^v\big)^2 \\
   & = &  B_{j+1}B_j+c^2b_1^h(a_{j-1}^v)^2\big(B_{j+1}+c^2b_1^h(a_{j}^v)^2\big) \\
   & = & B_j\big(B_{j+1}+c^2b_1^h(a_{j-1}^v)^2\big) \\
  \end{array}
$$
The second assertion is true.\\
The third assertion follows from Equations (\ref{sum 1/B_j}) and (\ref{sumB_j}).
\end{proof}
\textbf{Proof of Theorem \ref{Laplacian on generalized}}\\
The proof of the theorem is a very long calculation. that will be omitted here.\\

Now, we calculate the Laplacian of the lifts $f_1^h$ of $f_1$, using Lemma \ref{orthonormal basis}.
$$
\triangle(f_1^h)=\sum_{j=1}^{m_1+m_2}G_{f_1f_2}(\nabla_{u_j}grad(f_1^h),u_j)
$$
$$
=\sum_{j=1}^{m_1}G_{f_1f_2}(\nabla_{u_j}grad(f_1^h),u_j)+\hskip -0.4 cm\sum_{j=m_1+1}^{m_1+m_2}G_{f_1f_2}(\nabla_{u_j}grad(f_1^h),u_j)
$$
\begin{equation}\label{sum}
\hskip 0.3 cm=\big(\frac{1}{f_2^v}\big)^2\sum_{j=1}^{m_1}G_{f_1f_2}(\nabla_{e_j^h}grad(f_1^h),e_j^h)
+\hskip -0.4 cm\sum_{j=m_1+1}^{m_1+m_2}\frac{1}{\|u'\|^2}G_{f_1f_2}(\nabla_{u'_j}grad(f_1^h),u'_j)
\end{equation}
Calculate, the first term on the right-hand side of the last equation above
$$
\begin{array}{ccl}
G_{f_1f_2}(\nabla_{e_j^h}grad(f_1^h),e_j^h) \!\!\!\!&=\!\!\!\!\! & e_j^h(e_j^h(f_1^h))-(\nabla_{e_j^h}e_j^h)(f_1^h) \\
  &=\!\!\!\!\!&\left(g_1(\nabla{\hskip -0.23cm^{^{^{1}}}}_{e_j}gradf_1,e_j)\right)^h-f_2^v\left(K_{f_1}(e_j,e_j)\right)^hgrad(f_2^v)(f_1^h) \\
  &=\!\!\!\!\!\!\!&\left(g_1(\nabla{\hskip -0.23cm^{^{^{1}}}}_{e_j}gradf_1,e_j)\right)^h\left\{1-\!cf_1^hf_2^vgrad(f_2^v)(f_1^h)\right\}\\

  &&+f_2^vgrad(f_2^v)(f_1^h)\left\{1-c(e_j(f_1))^2\right\}^h.
 \end{array}
$$
From this formula and a straightforward calculation, we obtain
\begin{equation}\label{triangle_1}
\sum_{j=1}^{m_1}G\hskip -0.1cm{_{_{_{f_1f_2}}}}\!\!\!(\nabla_{e_j^h}grad(f_1^h),e_j^h)=\!\!\frac{1}{(1-c^2b_1^hb_2^v)}
\left\{\left(\triangle_1(f_1)\right)^h\!\!+\!\!\frac{cb_1^hb_2^v}{f_1}(cb_1^h-m_1)\right\}.
\end{equation}
Calculate, the second term on the right-hand side of Equation (\ref{sum}).
Straightforward calculation using Lemmas \ref{X_j and e_j} and \ref{sum 1/B_j and B_j} gives
$$
\sum_{j=m_1+1}^{m_1+m_2}\frac{1}{\|u'\|^2}G_{f_1f_2}(\nabla\hskip-0.1cm_{_{_{u'_j}}}\hskip -0.1cm grad(f_1^h),u'_j)
= \frac{c^2b_2^v}{(f_2^v)^2(1-c^2b_1^hb_2^v)}G_{f_1f_2}(\nabla\hskip-0.3cm_{_{_{_{_{(gradf_1)^h}}}}}\hskip-0.8cm grad(f_1^h),(gradf_1)^h)
$$
$$
-\frac{2c}{f_1^hf_2^v(1-c^2b_1^hb_2^v)}G_{f_1f_2}(\nabla\hskip-.3cm_{_{_{_{(gradf_1)^h}}}}\hskip-0.8cm grad(f_1^h),(gradf_2)^v)
+\big(\frac{1}{f_1^h}\big)^2\big[\sum_{j=m_1+1}^{m_1+m_2}\frac{1}{\|u'_j\|^2}\big\{
$$
$$
\big(\frac{c^2b_1^ha_j^v}{1-c^2b_1^hA_j}\big)^2G_{f_1f_2}
(\nabla\hskip-0.2cm_{_{_{_{T_j^v}}}}\hskip -0.1cm grad(f_1^h),T_j^v)+\frac{2c^2b_1^ha_j^v}{1-c^2b_1^hA_j}G_{f_1f_2}
(\nabla\hskip-0.2cm_{_{_{_{T_j^v}}}}\hskip -0.1cm grad(f_1^h),e_j^v)
$$
$$
+G_{f_1f_2}
(\nabla\hskip-0.2cm_{_{_{_{e_j^v}}}}\hskip -0.1cm grad(f_1^h),e_j^v)\}].
$$
So,
$$
\sum_{j=m_1+1}^{m_1+m_2}\frac{1}{\|u'\|^2}G_{f_1f_2}(\nabla\hskip-0.1cm_{_{_{u'_j}}}\hskip -0.1cm grad(f_1^h),u'_j)=
\frac{c^2b_2^v}{(f_2^v)^2(1-c^2b_1^hb_2^v)}G_{f_1f_2}(\nabla\hskip-0.3cm_{_{_{_{_{(gradf_1)^h}}}}}\hskip-0.8cm grad(f_1^h),(gradf_1)^h)
$$
$$
-\frac{2c}{f_1^hf_2^v(1-c^2b_1^hb_2^v)}G_{f_1f_2}(\nabla\hskip-.3cm_{_{_{_{(gradf_1)^h}}}}\hskip-0.8cm grad(f_1^h),(gradf_2)^v)
-\frac{1}{(f_1^h)^2\big(1-c^2b_1^hb_2^v\big)}\big\{(1-c^2b_1^hb_2^v)\hskip-.3cm\sum_{m_1+1}^{m_1+m_2}
\hskip-.2cm\nabla_{e_j^v}e_j^v(f_1^h)
$$
$$
+c^2b_1^h\sum_{j=m_1+1}^{m_1+m_2}
\big(a_j^v\big)^2\nabla_{_{e_j^v}}e_j^v(f_1^h)+2c^2b_1^h\hskip -1cm\sum_{m_1+1\leq i<j\leq m_1+m_2}
\hskip -1cm a_i^va_j^v\nabla_{e_i^v}e_j^v(f_1^h)\big\}.
$$
Since $\nabla$ is torsion-free and $[e_j^v,e_j^v](f_1^h)=0$, we deduce that
$$
2\hskip -1cm\sum_{m_1+1\leq i<j\leq m_1+m_2}
\hskip -1cm a_i^va_j^v\nabla_{e_i^v}e_j^v(f_1^h)+\sum_{m_1+1}^{m_1+m_2}
\big(a_j^v\big)^2\nabla_{_{e_j^v}}e_j^v(f_1^h)=\hskip -1cm\sum_{m_1+1\leq i,j\leq m_1+m_2}
\hskip -1cm a_i^va_j^v\nabla_{e_i^v}e_j^v(f_1^h).
$$
So
$$
\sum_{j=m_1+1}^{m_1+m_2}\frac{1}{\|u'\|^2}G_{f_1f_2}(\nabla\hskip-0.1cm_{_{_{u'_j}}}\hskip -0.1cm grad(f_1^h),u'_j)=\frac{c^2b_2^v}{(f_2^v)^2(1-c^2b_1^hb_2^v)}G_{f_1f_2}(\nabla\hskip-0.3cm_{_{_{_{_{(gradf_1)^h}}}}}\hskip-0.8cm grad(f_1^h),(gradf_1)^h)
$$
$$
-\frac{2c}{f_1^hf_2^v(1-c^2b_1^hb_2^v)}G_{f_1f_2}(\nabla\hskip-.3cm_{_{_{_{(gradf_1)^h}}}}\hskip-0.8cm grad(f_1^h),(gradf_2)^v)
-\big(\frac{1}{f_1^h}\big)^2\big\{\sum_{j=m_1+1}^{m_1+m_2}\hskip-.2cm\nabla_{e_j^v}e_j^v(f_1^h)
$$
$$
+\frac{c^2b_1^h}{1-c^2b_1^h2b_2^v}\hskip -0.8cm\sum_{_{_{m_1+1\leq i,j\leq m_1+m_2}}}
\hskip -1cm a_i^va_j^v\nabla_{e_i^v}e_j^v(f_1^h)\big\}.
$$
Since $\nabla$ is compatible with $G_{f_1f_2}$ and~
$$
\sum\hskip-1cm_{_{_{_{_{_{_{_{m_1+1\leq i,j\leq m_1+m_2}}}}}}}}
\hskip -1.3cm a_i^va_j^v\nabla_{e_i^v}e_j^v(f_1^h)=-G_{f_1f_2}(\nabla\hskip -0.4cm_{_{_{_{_{_{(gradf_2)^v}}}}}}
\hskip -0.7cm grad(f_1^h),(gradf_2)^v)
$$
then
$$
\sum_{j=m_1+1}^{m_1+m_2}\frac{1}{\|u'\|^2}G_{f_1f_2}\left(\nabla\hskip-0.1cm_{_{_{u'_j}}}\hskip -0.1cm grad(f_1^h),u'_j\right)
=\frac{c^2b_2^v}{(f_2^v)^2(1-c^2b_1^hb_2^v)}G_{f_1f_2}\left(\nabla\hskip-0.3cm_{_{_{_{_{(gradf_1)^h}}}}}
\hskip-0.8cm grad(f_1^h),(gradf_1)^h\right)
$$
$$
-\frac{2c}{f_1^hf_2^v(1-c^2b_1^hb_2^v)}G_{f_1f_2}\left(\nabla\hskip-.3cm_{_{_{_{(gradf_1)^h}}}}
\hskip-0.8cm grad(f_1^h),(gradf_2)^v\right)
-\big(\frac{1}{f_1^h}\big)^2\sum_{j=m_1+1}^{m_1+m_2}\hskip-.2cm\nabla_{e_j^v}e_j^v(f_1^h)
$$
$$
+\frac{c^2b_1^h}{(f_1^h)^2(1-c^2b_1^hb_2^v)}G_{f_1f_2}\left(\nabla\hskip -0.4cm_{_{_{_{_{_{(gradf_2)^v}}}}}}
 \hskip -0.7cmgrad(f_1^h),(gradf_2)^v\right).
$$
Using Proposition \ref{generalized}, we obtain
$$
\Delta(f_1^h)=\frac{1}{f_2^v(1-c^2b_1^hb_2^v)}\left\{\frac{1}{f_2^v}\big(\Delta_1(f_1)\big)^h-\frac{cb_1^h}{f_1^h}\big(\Delta_2(f_2)\big)^v
+\frac{b_1^h\left(c(1-m_1)b_2^v+m_2\right)}{f_1^hf_2^v}\right\}
$$
$$
+\frac{c^2}{2f_2^v(1-c^2b_1^hb_2^v)^2}\left\{\frac{b_2^v}{f_2^v}\left(gradf_1(b_1)\right)^h
-\frac{c(b_1^2)^h}{f_1^h}\left(gradf_2(b_2)\right)^v\right\}.
$$
For the Laplacian of $f_2^v$, just take $\{w_{_1},...,w_{_{m_2}},w_{_{m_2+1}},...,w_{_{m_2+m_1}}\}$ the local frame field
with repect to the metric $G_{f_{_1}f_{_2}})$, where
\begin{equation}
  W'_j=\left\{
     \begin{array}{ll}
\frac{1}{f_{_1}^v}e_j'^v, &j\in\{1,...,m_{_2}\}; \\
\frac{1}{f_2^h}\big(\frac{ c^2b_2^va_j^h}{(1-c^2b_{_2}^vA_j^h)}T_j^h+e_j'^h\big)-\frac{ca_j^h}{f_{_1}^h(1-c^2b_2^vA_j^h)}(gradf_2)^v,& j\in\{m_2+1,.,m_2+m_1\}.
     \end{array}
   \right.
\end{equation}
And
$$
w_j=\frac{1}{\|w'_j\|}w'_j, \hskip 0.1cm \|w_j'\|^2=\frac{1-c^2b_2^vA_{j+1}^h}{1-c^2b_2^vA_{j}^h}, \hskip 0.1cm
A_j=\!\!\!\!\!\sum_{i=m_{_2}+1}^{j-1}\!\!\!\!a_{_i}^2,
\hskip 0.12cm T_j=\!\!\!\!\!\sum_{i=m_{_2}+1}^{j-1}\!\!\!\!\!a_ie'_i,\hskip 0.1cm a_j=e_j(f_{_1}).
$$
such that $\{e'_{_1},...,e'_{_{m_{_2}}}\}$ is the local frame field
with respect to the metric $g_{_2}$ and $\{e'_{_{m_2+1}},...,e'_{_{m_2+m_1}}\}$ is the local frame field
with respect to the metric $g_{_1}$. Then
$$
\Delta(f_2^v)=\frac{1}{f_1^h(1-c^2b_1^hb_2^v)}\left\{-\frac{cb_2^v}{f_2^v}\big(\Delta_1(f_1)\big)^h+\frac{1}{f_1^h}\big(\Delta_2(f_2)\big)^v
+\frac{b_2^v\left(c(1-m_2)b_1^h+m_1\right)}{f_1^hf_2^v}\right\}
$$
$$
+\frac{c^2}{2f_1^h(1-c^2b_1^hb_2^v)^2}\left\{-\frac{c(b_2^2)^v}{f_2^v}\left(gradf_1(b_1)\right)^h
+\frac{b_1^h}{f_1^h}\left(gradf_2(b_2)\right)^v
\right\}.
$$
\begin{corollary}
If $f_1$ and $f_2$ are two harmonic functions, then $f_1^h$ (resp. $f_2^v$) is harmonic if and only if
$$
\frac{b_1^h\left(c(1-m_1)b_2^v+m_2\right)}{f_1^hf_2^v}+\frac{c^2}{2(1-c^2b_1^hb_2^v)}
\left\{\frac{b_2^v}{f_2^v}\left(gradf_1(b_1)\right)^h
-\frac{c(b_1^2)^h}{f_1^h}\left(gradf_2(b_2)\right)^v\right\}=0,
$$
$$
\left(resp.~\frac{b_2^v\left(c(1-m_2)b_1^h+m_1\right)}{f_1^hf_2^v}\!+\!\frac{c^2}{2(1-c^2b_1^hb_2^v)}
\!\left\{\!\frac{b_1^h}{f_1^h}\left(gradf_2(b_2)\right)^v
\!\!-\!\frac{c(b_2^2)^v}{f_2^v}\left(gradf_1(b_1)\right)^h\!\right\}\!=\!0\!\right).
$$

\end{corollary}
\begin{proof}
   As a direct consequence of Theorem \ref{Laplacian on generalized}.
\end{proof}
\begin{remark}
   1/ If for all $i\in\{1,2\}$, the gradient of $f_i$ is parallel with respect to $\nabla{\hskip -0.2cm^{^{^{i}}}}$, then
    $$
    \Delta(f_1^h)=\frac{(1-m_1)cb_1^hb_2^v+m_2b_1^h}{f_1^h(f_2^v)^2(1-c^2b_1^hb_2^v)} ~\text{and}~
    ~\Delta(f_2^v)=\frac{(1-m_2)cb_1^hb_2^v+m_1b_2^v}{(f_1^h)^2f_2^v(1-c^2b_1^hb_2^v)}.
    $$
    2/ If $\varphi_i\in C^{\infty}(M_i)$ $(i=1,2)$, then it is easy to calculate the $\triangle(\varphi_1^h)$
     and $\triangle(\varphi_2^v)$ by using of course Lemmas \ref{orthonormal basis}, \ref{X_j and e_j} and \ref{sum 1/B_j and B_j}.\\
     3/ We can calculate also, the bi-Laplacian of $\varphi_1^h$ and $\varphi_2^v$.
\end{remark}

\section{Other remarkable metric tensor on a product manifold.}
Let $c$ be an arbitrary real number and let $g_i$, $(i=1,2)$ be a Riemannian metric
tensors on $M_i$. Given a smooth positive function $f_i$ on $M_i$,
we define a metric tensor field on $M_1\times M_2$ by
\begin{equation}\label{Other metric}
        h_{_{f_1,f_2}}=\pi_1^*g_{_{_1}}+(f_1^h)^2\pi_2^*g_{_{_2}}
        +\frac{c^2}{2}(f_2^v)^2 df_1^h\odot df_1^h.
\end{equation}
Where $\pi_i$, $(i=1,2)$ is the projection of $M_{_1}\times M_{_2}$ onto $M_{_i}$.\\
It is the unique metric tensor such that for any $X_i,Y_i\in\Gamma(TM_i)$, $(i=1,2)$
\begin{equation} \label{equivalent Other metric}
\left\{
  \begin{array}{ccl}
  h_{_{f_1,f_2}}(X_1^h,Y_1^h)=g_1(X_1,Y_1)^h+c^2(f_2^v)^2X_1(f_1)^hY_1(f_1)^h  & \\
 h_{_{f_1,f_2}}(X_2^v,Y_2^v)=(f_1^h)^2g_2(X_2,Y_2)^v&\\
 h_{_{f_1,f_2}}(X_1^h,Y_2^v)=h_{_{f_1,f_2}}(X_2^v,Y_1^h)=0&

  \end{array}
\right.
\end{equation}

\subsection{The Levi-Civita Connection}
\begin{lemma}\label{Other grad}
Let $(M_i,g_i)$, $(i=1,2)$ be a Riemannian manifold. The gradient of the lifts $\varphi_1^h$ of $\varphi_{1}$
 and $\varphi_2^v$ of $\varphi_{2}$ to $M_1\times_{f_1,f_2}M_2$  w.r.t. $h_{f_1,f_2}$ are
\begin{equation}
 grad(\varphi_1^h)\!=\!(grad\varphi_1)^h-\frac{c^2(f_2^v)^2(grad\varphi(f_1))^h}{1+c^2(f_2^v)^2b_1^h}(gradf_1)^h,
\end{equation}
\begin{equation}
 grad(\varphi_2^v)\!=\!\frac{1}{(f_1^h)^2}(grad\varphi_2)^v,
\end{equation}
where $b_1=\| gradf_1\|^2$.
\end{lemma}
\begin{proof}
It suffices to observe that
$$
Z_1(f_1)^h=h_{_{f_1,f_2}}(\frac{1}{1+c^2(f_2^v)^2b_1^h}(gradf_1)^h,Z_1^h),
$$
and so,
$$
Z_1(\varphi_1)^h=h_{_{f_1,f_2}}((grad\varphi_1)^h-\frac{c^2(f_2^v)^2(grad\varphi(f_1))^h}{1+c^2(f_2^v)^2b_1^h}(gradf_1)^h,Z_1^h).
$$
Therefor, the result follows by Equation (\ref{equivalent Other metric}) and Lemma \ref{lift}.
\end{proof}

let us compute the Levi-Civita connection of $M_1\times_{f_1f_2}M_2$ associated with
 the metric $h_{f_1f_2}$ in terms of the Levi-Civita connections $\nabla{\hskip -0.2cm^{^{^{1}}}}$
 and $\nabla{\hskip -0.2cm^{^{^{2}}}}$ associated with the metrics $g_1$ and $g_2$
respectively.
\begin{proposition}\label{Other generalized}
Let $(M_i,g_{i})$, $(i=1,2)$ be a Riemannian manifold. Then we have

$$
  \nabla _{\!X_1^h}\!Y_1^{h}=(\nabla{\hskip -0.2cm^{^{^{1}}}}_{\!X_1}\!Y_1)^h
  +\frac{(cf_2^v)^2 H^{f_1}(X_1,Y_1)^h}{1+(cf_2^v)^2b_1^h}(gradf_1)^h
$$
  \begin{equation}\label{Other Horisontal  }
 -c^2f_2^v(X_1(\ln f_1)Y_1(\ln f_1))^h(gradf_2)^v
 \end{equation}
\begin{equation}\label{Other vertical }
  \nabla _{X_2^v}Y_2^{v}=(\nabla{\hskip -0.2cm^{^{^{2}}}}_{X_2}Y_2)^v
  -\frac{f_1^hg_2(X_2,Y_2)^v}{1+c^2(f_2^v)^2b_1^h}(gradf_1)^h
\end{equation}
\begin{equation}\label{Other milange}
 \begin{array}{rl}
 \hskip 0.3cm\nabla _{X_1^h}Y_2^{v}= \nabla _{Y_2^v}X_1^{h}&=\frac{c^2f_2^vY_2(f_2)^vX_1(f_1)^h}{(1+c^2(f_2^v)^2b_1^h)}(gradf_1)^h
 +\big(X_1(\ln f_{_{1}})\big)^hY_2^v,
 \end{array}
 \end{equation}
Where $H^{f_1}$ is the Hessian of $f_1$.
\end{proposition}
\begin{proof}
It follows directly from Koszul formula and Equation (\ref{equivalent Other metric}).
\end{proof}

\subsection{The Laplacian of the lifts to $M_1$ and $M_2$}
\begin{theorem}\label{Laplacian on Other generalized}
On a generalized warped product $(M_{_1}\times_{f_{_1}f_{_2}} M_{_2}, h_{f_{_1}f_{_2}})$ with $m_{_1}=dim M_{_1}$
and $m_{_2}=dim M_{_2}$, Let $\varphi_{_1}:M_1 \rightarrow \mathbb{R}$
and $\varphi_{_2}:M_1 \rightarrow \mathbb{R}$ be a smooth functions. Then the Laplacian
of the horizontal lift $\varphi_{_1}\circ \pi_1$ of $\varphi_{_1}$  (resp. vertical lift $\varphi_{_2}\circ \pi_2$ of $\varphi_{_2}$ )
to $M_{_1}\times_{f_1f_2} M_{_2}$ is given by
\begin{equation}\label{Lp on Other metric}
\Delta(\varphi_1^h)=\Delta(\varphi_1)^h+\frac{m_2(gradf_1(\varphi_1))^h}{f_1^h(1+(cf_2^v)^2b_1^h)}
-\frac{(cf_2^v)^2}{1+(cf_2^v)^2b_1^h}
\left\{
\begin{array}{lll}
&\\
\hskip-0.1cm(gradf_1(\varphi_1))^h\Delta(f_1)^h&\\
&\\
\end{array}
\right.
\end{equation}
$$
\hskip 0.5cm+H^{\varphi_1}(gradf_1,gradf_1)^h
-\frac{(cf_2^v)^2(gradf_1(\varphi_1))^h}{1+(cf_2^v)^2b_1^h}H^{f_1}(gradf_1,gradf_1)^h
\hskip-0.1cm\left\}
            \begin{array}{lll}
               &\\
                &  \\
               &
            \end{array}
          \right.
$$
\begin{equation}
\Delta(\varphi_2^v)=\frac{1}{(f_1^h)^2}\left\{\Delta(\varphi_2)^v+
\frac{c^2f_2^vb_1^h(gradf_2(\varphi_2))^v}{1+(cf_2^v)^2b_1^h}\right\}.
\end{equation}
Where $b_1=\| gradf_1\|^2$.
\end{theorem}
\begin{lemma}\label{Other orthonormal basis}
On $(M_{_1}\times_{f_{_1}f_{_2}} M_{_2}, h_{f_{_1}f_{_2}})$, if $\{e_{_1},...,e_{_{m_{_1}}}\}$ is the local frame field
with respect to the metric $g_{_1}$ and $\{e_{_{m_1+1}},...,e_{_{m_1+m_2}}\}$ is the local frame field
with respect to the metric $g_{_2}$, then $\{u_{_1},...,u_{_{m_1}},u_{_{m_1+1}},...,u_{_{m_1+m_2}}\}$ is the local frame field
with repect to the metric $h_{f_{_1}f_{_2}}$, where
\begin{equation}
  u'_i\!=\!\left\{
     \begin{array}{cc}
-\frac{(cf_2^v)^2a_i^h}{1+(cf_2^v)^2B_i^h}T_i^h+e_i^h, &i\in\{1,...,m_{_1}\}; \\
&\\
\frac{1}{f_1^h}e_i^v,& i\in\!\{m_1+1,.,m_1+m_2\}.
     \end{array}
   \right.
\end{equation}
And for $i\in\{1,...,m_1\}$,
$$
u_i\!\!=\!\!\frac{1}{\|u'_i\|}u'_i, \hskip 0.1cm \|u_i'\|^2\!=\!\frac{1+(cf_2^v)^2B_{i+1}^h}{1+(cf_2^v)^2B_i^h},
\hskip 0.1cm B_i\!=\!\sum_{j=1}^{i-1}\!\!a_{_j}^2,
 \hskip 0.12cm T_i\!=\!\sum_{j=1}^{i-1}\!\!a_ie_i,\hskip 0.1cm
a_i=e_i(f_{_1}).
$$
\end{lemma}
\begin{proof}
For the proof of the lemma it is actually almost the most interesting result because it provides
an algorithm for constructing $\{u_{_1},...,u_{_{m_1}},u_{_{m_1+1}},...,u_{_{m_1+m_2}}\}$
from the family $\{e_{_1},...,e_{_{m_{_1}}}\}$ et $\{e_{_{m_1+1}},...,e_{_{m_1+m_2}}\}$.\\
To do so, we use a limited recurrence (The Gram schmidt process)(see Lemma \ref{orthonormal basis}).\\
\end{proof}
\begin{remark}
With the notations above, we have \\
1) $T_{1}$ is the zero vector field on $M_{_1}$, $B_{1}$ is the zero function on $M_{_1}$ and
$A_{m_1+1}$ is the care of the gradient of $f_{_1}$.\\
2)For any $i\in\{1,...,m_1+1\}$
\begin{equation}\label{u_j}
\left\{
  \begin{array}{lll}
  T_i(f_{_1})=A_i=g_(T_i,T_i), &  \\
   u'_i(f_{_1}^h)=\frac{a_i^h}{1+(cf_2^v)^2B_i^h} . &
  \end{array}
\right.
\end{equation}
\end{remark}

\begin{lemma}\label{Other sum 1/B_j and B_j}
  With the notations above, we have, for all $j\in\{1,...,m_1\}$
\begin{equation}\label{Other sum 1/B_j}
(cf_2^v)^4(a_j^h)^2\left(\sum_{i=j+1}^{m_1}\frac{(a_i^h)^2}{(1+(cf_2^v)^2B_i^h)(1+(cf_2^v)^2B_{i+1}^h)}\right)
+\frac{1+(cf_2^v)^2B_j^h}{1+(cf_2^v)^2B_{j+1}^h}=1-\frac{(cf_2^va_j^h)^2}{1+(cf_2^v)^2b_1^h},
\end{equation}
\begin{equation}\label{ Other sumB_j}
(cf_2^v)^2\left(\sum_{i=j+1}^{m_1}\frac{(a_i^h)^2}{(1+(cf_2^v)^2B_i^h)(1+(cf_2^v)^2B_{i+1}^h)}\right)
-\frac{1}{1+(cf_2^v)^2B_{j+1}^h}=\frac{-1}{1+(cf_2^v)^2b_1^h}
\end{equation}
 \end{lemma}
\begin{proof}
The proof is the following partial analogue of Lemma \ref{sum 1/B_j and B_j}.
\end{proof}

\textbf{Proof of Theorem \ref{Laplacian on Other generalized}}\\
The proof of the theorem is a very long calculation. that will be omitted here.\\

Now, we calculate the Laplacian of the lifts $\varphi_1^h$ of $\varphi_1$, using Lemma \ref{Other orthonormal basis}.
$$
\triangle(\varphi_1^h)=\sum_{j=1}^{m_1+m_2}h_{f_1f_2}(\nabla_{u_j}grad(\varphi_1^h),u_j)
$$
$$
=\sum_{j=1}^{m_1}h_{f_1f_2}(\nabla_{u_j}grad(\varphi_1^h),u_j)+
\hskip -0.4 cm\sum_{j=m_1+1}^{m_1+m_2}h_{f_1f_2}(\nabla_{u_j}grad(\varphi_1^h),u_j)
$$
\begin{equation}\label{Other sum}
\hskip 0.3 cm=\hskip -0.1 cm\sum_{j=1}^{m_1}\frac{1}{\|u'\|^2}h_{f_1f_2}(\nabla_{u'_j}grad(\varphi_1^h),u'_j)
+\big(\frac{1}{f_1^h}\big)^2\!\!\!\!\sum_{j=m_1+1}^{m_1+m_2}
\!\!\!h_{f_1f_2}(\nabla_{e_j^v}grad(\varphi_1^h),e_j^v)
\end{equation}
Calculate, the seconde term on the right-hand side of the last equation above
$$
\big(\frac{1}{f_1^h}\big)^2\!\!\!\!\sum_{j=m_1+1}^{m_1+m_2}
\!\!\!h_{f_1f_2}(\nabla_{e_j^v}grad(\varphi_1^h),e_j^v)=
\frac{m_2(gradf_1(\varphi_1))^h}{f_1^h(1+(cf_2^v)^2b_1^h)} .
$$
Calculate, the first term on the right-hand side of Equation (\ref{Other sum}).
Straightforward calculation using Lemmas \ref{Other sum 1/B_j and B_j} gives
$$
\sum_{j=1}^{m_1}\frac{1}{\|u'\|^2}h_{f_1f_2}(\nabla_{u'_j}grad(\varphi_1^h),u'_j)
=\!\sum_{j=1}^{m_1}h_{f_1f_2}(\nabla_{e_j^h}grad(\varphi_1^h),e_j^h)
$$
$$
-\frac{(cf_2^v)^2}{1+(cf_2^v)^2b_1^h}
\!\sum_{1\leq i,j\leq m_11}\hskip -0.4 cm a_i^ha_j^hh_{f_1f_2}(\nabla_{e_j^h}grad(\varphi_1^h),e_j^h)
$$
$$
=\!\sum_{j=1}^{m_1}h_{f_1f_2}(\nabla_{e_j^h}grad(\varphi_1^h),e_j^h)-\frac{(cf_2^v)^2}{1+(cf_2^v)^2b_1^h}
h_{f_1f_2}(\nabla\hskip -0.2 cm_{_{_{_{(gradf_1)^h}}}}\hskip -0.8 cmgrad(\varphi_1^h),(gradf_1)^h).
$$
Using Proposition \ref{Other generalized}, we obtain Equation (\ref{Lp on Other metric}).\\
The seconde assertion is similar.
\begin{corollary}Let $(M_i,g_i)$ $(i=1,2)$ be a connected riemannian manifolds.
If $f_1$ is a harmonic function such that $gradf_1\neq 0$, then $f_1^h$ is harmonic if and only if
$$
 c\neq 0,~f_2 \text{~is~ a constant ~function ~and~}~ H^{f_1}(gradf_1,gradf_1)=m_2b_1(b_1+\frac{1}{c^2f_2^2}),
$$
If $f_2$ is harmonic function, then $f_2^v$ is harmonic if and only if
$$
c=0~\text{or}~(f_1~ \text{or}~ f_2~ \text{is a constant function}).
$$

\end{corollary}
\begin{proof}
   As a direct consequence of Theorem \ref{Laplacian on generalized}.
\end{proof}

\subsection{The Curvature tensors}
Let $\mathcal{R}{\hskip -0.2cm^{^{^{i}}}}$ $(i=1,2)$ and $\mathcal{R}$ be the Riemannian curvature tensors
 with respect to $g_{_{_i}}$ and $g_{_{f_1f_2}}$ respectively.
In the following proposition, we express the curvature $\mathcal{R }$ of the connection
$\nabla$ in terms of the warping functions $f_1,f_2$ and the curvatures $\mathcal{R}{\hskip -0.2cm^{^{^{1}}}}$
and $\mathcal{R}{\hskip -0.2cm^{^{^{2}}}}$ of $\nabla{\hskip -0.2cm^{^{^{1}}}}$
 and $\nabla{\hskip -0.2cm^{^{^{2}}}}$ respectively.
\begin{proposition}\label{Other curvature generalized}
Let $(M_{_i},g_{_{_i}})$, $(i=1,2)$ be a connected Riemannian manifold and let
$f_1\in C^{\infty}(M_1)$ be a non-constant positive function.
Assume that the gradient of $f_i$ is parallel with respect to
$\nabla{\hskip -0.2cm^{^{^{i}}}}$ $(i=1,2)$. Then for any
${X_i},{Y_i},{Z_i}\in\Gamma(TM_{_i})$
$(i=1,2)$ we have
$$
\begin{array}{ll}
\mathcal{R}({X_1}^h,{Y_1}^h){Z_1}^h&\!\!\!\!=(\mathcal{R}^{^1}({X_1},{Y_1}){Z_1})^h,\\
&\\
\mathcal{R}({X_2}^v,{Y_2}^v){Z_2}^v&\!\!\!\!=(\mathcal{R}^{^2}({X_2},{Y_2}){Z_2})^v
-\frac{b_1}{1+(cf_2^v)^2b_1}\left\{(X_2\wedge_{g_2} Y_2)Z_2\right\}^v\\
&\\
&\!\!\!\!+\frac{c^2f_1^hf_2^vb_1}{\left(1+(cf_2^v)^2b_1\right)^2}
\left\{\left((X_2\wedge_{g_2} Y_2)Z_2\right)(f_2)\right\}^v(gradf_1)^h,\\
&\\
\mathcal{R}({X_1}^h,{Y_1}^h){Z_2}^v&\!\!\!\!=0,\\
&\\
\mathcal{R}({X_2}^v,{Y_2}^v){Z_1}^h&\!\!\!\!= \frac{c^2f_2^vb_1(Z(f_1))^h}{f_1^h(1+(cf_2^v)^2b_1)}
\left\{(X_2\wedge_{g_2} Y_2)gradf_2\right\}^v   ,\\
&\\
\mathcal{R}({X_1}^h,{Y_2}^v){Z_1}^h&\!\!\!\!= \frac{c^2X_1(\ln f_1)^hZ_1(\ln f_1)^h Y_2(f_2)^v}{1+(cf_2^v)^2b_1}(gradf_2)^v,\\
&\\
\mathcal{R}({X_1}^h,{Y_2}^v){Z_2}^v&\!\!\!\!=\frac{c^2X_1(\ln f_1)^h}{1+(cf_2)^2b_1}\left\{
f_2^vb_1\left((gradf_2\wedge_{g_2}Y_2)Z_2\right)^v-\frac{f_1^h Y_2(f_2)^vZ_2(f_2)^v}{1+(cf_2)^2b_1}(gradf_1)^h\right\}.
\end{array}
$$
where the wedge product $(X_2\wedge_{g_2} Y_2)Z_2=g_2(Y_2,Z_2)X_2-g_2(X_2,Z_2)Y_2$.
\end{proposition}
\begin{proof}
Long but straightforward computations using Proposition (\ref{Other generalized})
and Lemma(\ref{Other grad}).
\end{proof}
As direct consequence of Proposition \ref{Other curvature generalized} we obtain
\begin{corollary}
Let $(M_{_i},g_{_{_i}})$, $(i=1,2)$ be a Riemannian manifold.
Assume that the gradient of $f_1$ is parallel with respect to
$\nabla{\hskip -0.2cm^{^{^{1}}}}$. If $(M_1\times M_2, h_{f_1f_2})$ is flat then
the base $(M_1,g_1)$ is flat and the fiber $(M_2,g_2)$ is space of constant sectional curvature $k=\frac{b_1}{1+(cf_2^v)^2b_1}$.
\end{corollary}

Now consider the Ricci curvature Ric of a generalized warped product, writing
(R$ic_1)^h$ for the lift (pullback by $\pi_1$) of the Ricci curvature of $M_1$, and similarly for
$(Ric_2)^v$.

\begin{proposition}\label{Other ricci curvature}
Under the same assumptions as in Proposition \ref{Other curvature generalized}, let
$\mathcal{R}ic_1$, $\mathcal{R}ic_2$ and $\mathcal{R}ic$ be the Ricci curvature tensors with respect to
$g_{_{_1}}$, $g_{_{_2}}$ and $h_{_{f_1f_2}}$ respectively. let ${X_1},{Y_1}\in\Gamma(TM_{_1})$
and ${X_2},{Y_2}\in\Gamma(TM_{_2})$, then we have
$$
\begin{array}{rl}
\mathcal{R}ic(X_1^h,Y_1^h)&\!\!\!\!=\mathcal{R}ic_1(X_1,Y_1)^h-\frac{c^2 b_2^v}{1+(cf_2^v)^2 b_1}X_1(\ln f_1)^hY_1(\ln f_1)^h
, \\
&\\
\mathcal{R}ic(X_1^h,Y_2^v)&\!\!\!\!=\frac{c^2(m_2-1)b_1f_2^v}{1+(cf_2)^2b_1}X_1(\ln f_1)^hY_2(f_2)^v,\\
&\\
\mathcal{R}ic(X_2^v,Y_2^v)&\!\!\!\!=\mathcal{R}ic_2(X_2,Y_2)^v+\frac{c^2b_1}{(1+(cf_2^v)^2b_1)^2} X_2(f_2)^vY_2(f_2)^v
-\frac{(m_2-1)b_1}{1+(cf_2^v)^2b_1}g_2(X_2,Y_2)^v.
  \end{array}
$$
Where $m_2=dim M_2$.
\end{proposition}
\begin{proof}
 Long but straightforward computations using Propositions
(\ref{Other generalized}, \ref{Other curvature generalized})
and Lemmas (\ref{Other grad}, \ref{Other orthonormal basis} and \ref{Other sum 1/B_j and B_j}).
\end{proof}
\begin{corollary}\label{Other Scalar curvature}
Under the same assumptions as in Proposition \ref{Other curvature generalized},
 let $\mathcal{S}_1$, $\mathcal{S}_2$ and $\mathcal{S}$ be the scalar curvature
with respect to $g_{_{_1}}$, $g_{_{_2}}$ and $g_{_{f_1f_2}}$ respectively.
Then the following equation holds
$$
\begin{array}{ccl}
  \mathcal{S}&=\mathcal{S}_1^h+\frac{1}{(f_1^h)^2}\mathcal{S}_2^v-\frac{m_2(m_2-1)b_1}{(f_1^h)^2(1+(cf_2)^2b_1)}.&
\end{array}
$$
\end{corollary}
\begin{proof}
Follows from Propositions  (\ref{Other generalized}, \ref{Other curvature generalized})
and Lemmas (\ref{Other grad}, \ref{Other orthonormal basis} and \ref{Other sum 1/B_j and B_j}).
\end{proof}

\begin{corollary}
Under the same assumptions as in Proposition \ref{Other curvature generalized}, let $(M_i,g_i)$ $(i=1,2)$
be a riemannian manifold with constant sectional curvature $k_i$. Then
$$
 \mathcal{S}(p_1,p_2)=m_1(m_1-1)k_1+ \frac{m_2(m_2-1)}{f_1(p_1)^2}\left(k_2-
 \frac{\|gradf_1\|_{p_1}}{1+(cf_2(p_2))^2\|gradf_1\|_{p_1}}\right).
$$
\end{corollary}
\begin{proof}
We know that if $(M_i,g_i)$ $(i=1,2)$ have constant sectional curvature $k_i$,
then $\mathcal{S}_i(p_i)=m_i(m_i-1)k_i$. By Corollary \ref{Other Scalar curvature} follows.
\end{proof}

\textbf{Acknowledgement:} A big part of this work was done at The Raphel Salem Laboratory of mathematics,
University of Rouen (France), Rafik Nasri would like to thank Raynaud de Fitte Paul,
Siman Raulot for very useful discussions and the Mathematic section for their hospitality.





\medskip
Received xxxx 20xx; revised xxxx 20xx.
\medskip

\end{document}